\documentclass{amsart}
\usepackage{amsthm,amsfonts,amsmath,amscd,amssymb,latexsym,epsfig}
\usepackage[title,titletoc]{appendix}
\newcommand{\gl}{{\mathfrak g \mathfrak l}}
\newcommand{\so}{{\mathfrak s \mathfrak o}}

\newcommand{\su}{{\mathfrak s  \mathfrak u}}

\newcommand{\cx}{{\mathbb C}}

\newcommand{\re}{\operatorname{Re}}
\newcommand{\im}{\operatorname{Im}}

\newcommand{\End}{\operatorname{End}}

\newcommand{\Ker}{\operatorname{Ker}}

\newcommand{\D}{{\partial}}

\newcommand{\supp}{\operatorname{supp}}

\newcommand{\Gr}{\operatorname{Gr}}

\newcommand{\Rat}{\operatorname{Rat}}

\newcommand{\wh}{\widehat}
\newcommand{\ol}{\overline}

\numberwithin{equation}{section}

\newtheorem{theorem}{Theorem}[section]

\newtheorem{lemma}[theorem]{Lemma}

\newtheorem{corollary}[theorem]{Corollary}

\newtheorem{proposition}[theorem]{Proposition}

\theoremstyle{remark}

\newtheorem{remark}[theorem]{Remark}

\newtheorem{definition}[theorem]{Definition}

\newtheorem{example}[theorem]{Example}

\newcommand{\oC}{{\mathbb{C}}}

\newcommand{\oH}{{\mathbb{H}}}

\newcommand{\oN}{{\mathbb{N}}}

\newcommand{\oP}{{\mathbb{P}}}

\newcommand{\oR}{{\mathbb{R}}}

\newcommand{\oZ}{{\mathbb{Z}}}

\newcommand{\sC}{{\mathcal{C}}}   
\newcommand{\sD}{{\mathcal{D}}}
\newcommand{\sE}{{\mathcal{E}}}
\newcommand{\sF}{{\mathcal{F}}}
\newcommand{\sG}{{\mathcal{G}}}   

\newcommand{\sJ}{{\mathcal{J}}}
\newcommand{\sK}{{\mathcal{K}}}
\newcommand{\sL}{{\mathcal{L}}}   
\newcommand{\sM}{{\mathcal{M}}}   
\newcommand{\sN}{{\mathcal{N}}}
\newcommand{\sO}{{\mathcal{O}}}

\newcommand{\sT}{{\mathcal{T}}}
\newcommand{\sU}{{\mathcal{U}}}
\newcommand{\sV}{{\mathcal{V}}}
\newcommand{\sW}{{\mathcal{W}}}

\newcommand{\fG}{{\mathfrak{g}}}

\begin{document}

\title{Pluricomplex geometry and hyperbolic monopoles}
\author{Roger Bielawski \& Lorenz Schwachh\"ofer}
\address{School of Mathematics\\
University of Leeds\\Leeds LS2 9JT\\ UK}
\address{Fakult\"at f\"ur Mathematik\\ TU Dortmund\\
	D-44221 Dortmund\\ Germany}


\begin{abstract} Motivated by  strong desire to understand the natural geometry of moduli spaces of hyperbolic monopoles, we introduce and study a new type of geometry: {\em pluricomplex geometry}. 
\end{abstract}

\maketitle

\thispagestyle{empty}

\tableofcontents

\section{Introduction\label{intro}}

In this paper, we shall be concerned with a new type of differential geometry, which we call {\em pluricomplex geometry}. It is a generalisation of hypercomplex ge\nolinebreak ometry: we still have a $2$-sphere of complex structures, but they no longer behave like unit imaginary quaternions. We still require, however,  that the $2$-sphere of complex structures determines a decomposition of the complexified tangent space as $\cx^{2n}\otimes\cx^2$.
\par
It turns out that the geometry of such structures is very rich and can be profitably studied from at least two points of view. On the one hand,  the pluricomplex geometry of a manifold $M$ is the same as a special type of  hypercomplex geometry on the bundle
$$ T_\Delta M=\{(v,\bar v); v\in T^\cx M\}.$$
In fact, an integrable pluricomplex structure (i.e. a $2$-sphere of integrable complex structures satisfying the above condition) on a manifold $M$ can be viewed as an integrable hypercomplex structure on an complex thickening $M^\cx$ of $M$, commuting with the tautological complex structure of $M^\cx$ (thus, the pluricomplex geometry of $M$ can be viewed as a {\em biquaternionic} geometry of $M^\cx$). It follows, remarkably enough, that  any integrable pluricomplex structure has an associated canonical torsion-free connection (generally without special holonomy). 
\par
From another point of view,  a pluricomplex structure on a vector space $V$ corresponds to a coherent sheaf $\sF$ on an algebraic curve $S$ in $\oP^1\times \oP^1$, satisfying certain vanishing conditions. Writing $V^{1,0}_\zeta$ and $V^{0,1}_\zeta$ for vectors of type $(1,0)$ and $(0,1)$ for the complex structure $J_\zeta$, $\zeta\in \oP^1$, we can describe the curve $S$ as
$$ S=\{(\zeta,\eta)\in \oP^1\times \oP^1;\; V^{1,0}_\zeta\cap V^{0,1}_{-1/\bar \eta}\neq 0\},$$
while the stalk of $\sF$ at $(\zeta,\eta)\in S$ is $V^\cx/(V^{1,0}_\zeta+V^{0,1}_{-1/\bar\eta})$.
\par
We call $S$ the {\em characteristic curve} and  $\sF$ the {\em characteristic sheaf} of the pluricomplex structure. Hypercomplex structures can be characterised by the property that their characteristic curve is the diagonal $\oP^1\subset \oP^1\times \oP^1$.
\par
This point of view allows us to consider pluricomplex manifolds which have twistor spaces, i.e. can be recovered as a family of complex curves in a complex manifold $Z$.  The curve corresponding to $m\in M$ will be the characteristic curve of the pluricomplex structure on $T_mM$, while its normal bundle in $Z$ will be the characteristic sheaf. We call pluricomplex structures with a twistor space {\em strongly integrable}. For such structures, we discuss vector bundles on $M$ induced from holomorphic bundles on $Z$. We show that a large chunk of the twistor theory for self-dual or quaternionic manifolds still goes through. In particular, a vector bundle on $Z$, trivial on each twistor curve, still leads to a canonical linear connection on the induced bundle over $M$. A generalisation of a different type says that we can expect to get canonical linear connections on vector bundles induced from bundles on $Z$, which have a particular degree on each twistor curve (this degree is $0$ for quaternionic manifolds). We remark that these twistor curves will usually have positive genus.
\par
Characteristic sheaves of pluricomplex structures are also interesting from a purely algebraic point of view. In the appendix we show that they satisfy a rather strong regularity condition: for any $p,q\in \oZ$ with $p+q=-2$,  $\sF(p,q)$ has trivial cohomology.
\par
Let us point out here that {\em complexified} pluricomplex structures coincide with objects studied by many authors and known variously as  GS-structures \cite{M},  paraconformal structures \cite{Ba-E} or $SL(2)$-webs \cite{JV}. The {\em real} pluricomplex geometry is, however, very different from all of these.
\par
Throughout the paper we are motivated by a particular example: the moduli space $\sM_{k,m}$ of (framed) $SU(2)$-monopoles of charge $k$ on the hyperbolic $3$-space with curvature $-1/m^2$ \cite{A,MS1,MS2,Nash}. It is well known that the moduli space of euclidean monopoles $\sM_k=\sM_{k,\infty}$ has a natural hyperk\"ahler metric, which is of great physical significance. The moduli space of hyperbolic monopoles is a deformation of $\sM_k$. Furthermore, it can be constructed via twistor methods (at least if $2m\in \oZ$), which leads one to expect a natural geometry. In \cite{Hit},  Hitchin  constructed a natural self-dual Einstein metric on the moduli space of charge $2$ {\em centred} hyperbolic monopoles. The general case, however, has resisted a solution. A significant progress has been made by O. Nash \cite{Nash}, who found a new twistorial construction of $\sM_{k,m}$ and described\footnote{See, however,   Remark \ref{Nash-gap}.} the complexification of the natural geometry of $\sM_{k,m}$, whatever it might be. In fact, Nash writes \cite[p.1672]{Nash}: ``{\em 
The geometry we have identified on the hyperbolic monopole moduli space appears to
be a real geometry whose complexification is very like the complexification of hyperk\"ahler
geometry but which is subtly different}".
\par
The original aim of the present work was to identify this real geometry. As we show in Section \ref{hyperbolic}, it is a strongly integrable pluricomplex geometry\footnote{Strictly speaking, we only show this for an open dense subset of $\sM_{k,m}$. Its complement has a weaker structure, which is pluricomplex if certain vanishing theorem holds: see  Section \ref{hyperbolic}.}. Thus, in particular, there is a natural torsion-free connection on $\sM_{k,m}$.

\bigskip

{\it Acknowledgements.} This work has been supported, at various stages, by the {\em Connect} programme of the Humboldt Foundation, and the Gambrinus Fellowship from TU Dortmund. The authors thank both institutions and the University of Leeds for making the collaboration possible. 
\par
 We also thank Michael Atiyah, Nigel Hitchin and Michael Singer for comments.


\bigskip

{\em Remark on terminology.} We work in the  complex analytic category. By a {\em curve} we mean a compact complex space $(S,\sO_S)$ (not necessarily smooth or reduced), such that $\dim_x S=1$ at every point $x\in S$. Except for a few occasions, we do not distinguish between vector bundles and corresponding sheaves of local sections.

\section{$2$-spheres of complex structures and sheaves on $\oP^1\times \oP^1$\label{spheres}}

Let $V$ be a $2n$-dimensional real vector space. Associated to $V$ is its twistor space $\sJ(V)\simeq GL(2n,\oR)/GL(n,\cx)$, i.e. the space of complex structures on $V$. For any complex structure $J\in \sJ(V)$ we have
\begin{equation} T_J\sJ(V)=\{A\in \End V;\: JA=-AJ\}.\label{TJ}\end{equation}
Left multiplication by $J$ on $T_J\sJ(V)$ squares to $-1$ and defines a canonical (integrable) complex structure on $\sJ(V)$ (in fact, $\sJ(V)$ is a complex symmetric space).
We denote by $\sV^{1,0}\to \sJ(V)$ the tautological $i$-eigenbundle, i.e. a subbundle of $\underline{V^\cx}= V^\cx\times\sJ(V)$ with fibre $\sV^{1,0}_J=V^{1,0}(J)=\{X-iJX; X\in V\}$. 
\par
We can also view $\sJ(V)$ as an open dense subset of $\Gr_n(V^\cx)$, so that $\sV^{1,0}$ is the usual tautological bundle on a Grassmannian.

We shall now study arbitrary holomorphic $\oP^1$-s of complex structures, i.e. holomorphic maps $K:\oP^1\to \sJ(V)$ such that $\im K$ is not a point. 
 For every $\zeta\in \oP^1$, let us write $J_\zeta=K(\zeta)$ and $V^{1,0}_\zeta=\{X-iJ_\zeta X; X\in V\}$ - the subspace  of vectors of type $(1,0)$ for $J_\zeta$.
\par
For any such $K$, we may consider simultaneously a second $\oP^1$ of complex structures, defined by $\hat{J}_\zeta=-J_{\sigma(\zeta)}$, where $\sigma:\oP^1\to \oP^1$ is the antipodal map, $\sigma(\zeta)=-1/\bar\zeta$. The resulting $\wh K:\oP^1\to \sJ(V)$ is also holomorphic. We denote by $ \wh{V}^{1,0}_\zeta$ the subspace of vectors of type $(1,0)$ for $\hat{J}_\zeta$. In other words
$$\wh{V}^{1,0}_\zeta= V^{0,1}_{-1/\bar\zeta}.$$
\par
Let $W$ denote the holomorphic bundle on $\oP^1\times \oP^1$, the fibre of which at $(\zeta,\eta)$ is $V^{1,0}_\zeta\oplus\wh{V}^{1,0}_\eta$. In other words 
\begin{equation} W=(K\circ \pi_1)^\ast  \sV^{1,0}\oplus (\wh K\circ \pi_2)^\ast  \sV^{1,0}, \label{E}\end{equation}
where $\pi_1,\pi_2:\oP^1\times \oP^1\to \oP^1$ are the two projections.
We have a natural map 
$$ W\to O\otimes V^\cx$$ 
of vector bundles on $\oP^1\times \oP^1$, which induces an injective map $\sW\to \sO\otimes V^\cx$ on the sheaves of sections.
We denote by $\sF$ the cokernel of this map, so that we have an exact sequence
\begin{equation}0\to \sW {\rightarrow}\sO\otimes V^\cx\to \sF\to 0. \label{F0}\end{equation}
From \eqref{E}:
\begin{equation} W\simeq  \bigoplus_{j=1}^{n}\sO(-p_j,0)\oplus \bigoplus_{j=1}^{n}\sO(0,-p_j)\label{E-split}\end{equation}
for some nonnegative integers $p_j$, so that we can view \eqref{F0} as the following resolution of $\sF$:
\begin{equation} 0\to \bigoplus_{j=1}^{n}\sO(-p_j,0)\oplus \bigoplus_{j=1}^{n}\sO(0,-p_j)\stackrel{M(\zeta,\eta)}{\longrightarrow} \sO^{\oplus 2n}\longrightarrow \sF\to 0.\label{F1}
\end{equation}
Here $M(\zeta,\eta)$ is a $2n\times 2n$-matrix, and $M_{ij}$ is a polynomial of degree $p_i$ in $\zeta$ if $j\leq n$, 
and a polynomial of degree $p_i$ in $\eta$ if $j> n$. It is clear that, as a set, the support of $\sF$ is 
$$ S=\{(\zeta,\eta)\in \oP^1\times \oP^1;\; \det M(\zeta,\eta)=0\}.$$
Its analytic structure can be, however, different. To describe it, let us write $\det M(\zeta,\eta)=\prod_{i=1}^s q_i(\zeta,\eta)^{k_i}$, where $q_i$ are irreducible polynomials. We define the {\em minimal polynomial} $p_M(\zeta,\eta)$ of $M$ as
$\prod_{i=1}^s q_i(\zeta,\eta)^{l_i}$, where 
\begin{eqnarray*}l_i & = & \max\{a_ib_i;\; \text{at a generic point, $M(\zeta,\eta)$ has a Jordan block of size $a_i$}\\ & & \text{with eigenvalue $q_i(\zeta,\eta)^{b_i}$}\}.\end{eqnarray*}
Let us also recall the notion of a $\sigma$-sheaf:
\begin{definition} Let $X$ be a complex manifold equipped with an antiholomorphic  involution $\sigma$. 
A $\sigma$-sheaf on $X$ is a coherent sheaf $\sG$, such that the linear space associated to $\sG$ is equipped with an antilinear automorphism $\sigma_{\sG}$ covering $\sigma$ and satisfying $\sigma_\sG^2={\rm Id}$. \label{ssigma}\end{definition}
On $\oP^1\times \oP^1$, we have a canonical antiholomorphic involution
\begin{equation} \sigma(\zeta,\eta)=\bigl(\sigma(\eta),\sigma(\zeta)\bigr)=(-1/\bar\eta,-1/\bar \zeta),\label{sigma}\end{equation}
the fixed point set of which is the antidiagonal $\ol\Delta=\{(\zeta,\sigma(\zeta))\}$. 
\par

We can now summarise the properties $S$ and $\sF$ as follows
\begin{proposition}\begin{itemize}
\item[(i)] 
The support of $\sF$ is the $\sigma$-invariant curve  $\bigl(S,\sO_{\oP^1\times \oP^1}/(p_M)\bigr)$.  Furthermore
$S=\bigl\{(\zeta,\eta)\in \oP^1\times \oP^1; \; V^{1,0}_\zeta\cap \wh{V}^{1,0}_\eta\neq 0\bigr\},$
and, consequently, $S$ does not intersect $\ol\Delta$.
\item[(ii)] $\sF$ is a  $\sigma$-sheaf on $\oP^1\times \oP^1$ such that $H^0(\sF(-1,-1))=H^1(\sF(-1,-1))=0$.
\end{itemize}
\label{S}
\end{proposition} 
\begin{proof} Part (i) is clear from the definition, once we recall that the complex subspace $(S,\sO_S)$ is the zero-locus of the annihilator (in $\sO_X$) of $\sF$. For part (ii), we observe, from the defining sequence \eqref{F0},  that both $\sW$ and $\sO\otimes V^\cx$ are naturally $\sigma$-sheaves and the map between them is compatible with the automorphisms $\sigma_\sW$ and $\sigma_{\sO\otimes V^\cx}$. Therefore $\sF$ is a $\sigma$-sheaf. Finally, for the vanishing of cohomology, we tensor \eqref{F1} with $\sO(-1,-1)$ and take the long exact sequence of  cohomology.
\end{proof}

\begin{corollary} Suppose that the  curve $S$ is integral. Then
the sheaf $\sF$ is locally free on $S$. \label{VB0}\end{corollary}
\begin{proof} The assumption implies that $\det M=(p_M)^r$ and, so, the nonzero stalks of $\sF$ have constant rank $r$. Therefore $\sF|_S$ is locally free of rank $r$. \end{proof}

\medskip

We shall now describe how to  obtain a $2$-sphere of complex structures from a sheaf. We begin with a   $1$-dimensional $\sigma$-sheaf $\sF$ on $\oP^1\times\oP^1$ such that $H^\ast(\sF(-1,-1))=0$ and $\ol\Delta\cap \supp\sF=\emptyset$.
\par
Let $\zeta\in \oP^1$. Since $(\zeta,-1/\bar{\zeta})\not\in S$,  we have an exact sequence
\begin{equation} 0\to \sF(-1,-1)\to  \sF(-1,0)\oplus \sF(0,-1)\to \sF\to 0,\label{original}\end{equation}
where the map is given by multiplication by a pair of global non-zero sections of $\sO(1,0)$ and $\sO(0,1)$, vanishing at $\zeta$ and $-1/\bar\zeta$, respectively. Thus, we obtain a decomposition
\begin{equation}H^0(\sF)\simeq H^0(\sF(-1,0))\oplus H^0(\sF(0,-1)) \label{decompose2}\end{equation}
where the summands are interchanged by $\sigma$, i.e. by the real structure of $H^0(\sF)$. This gives a complex structure on $V=H^0(\sF)^\sigma$ for every $\zeta\in \oP^1$. The subspace $V_\zeta^{1,0}$ is identified with the vector space of global sections of $\sF$, vanishing on $D_\zeta=\{\zeta\}\times \oP^1$, i.e. $V^{1,0}_\zeta=H^0(\sF[-D_\zeta])$.  Thus we obtain a holomorphic embedding $K:\oP^1\to \sJ(V)$. 
\par
Observe that the bundle $K^\ast \sV^{1,0}$ is isomorphic to $H^0(\sF(-1,0))\otimes \sO(-1)$. The isomorphism is realised explicitly  via the map:
 $H^0(\sF(-1,0))\otimes \sO(-1)\to H^0(\sF)\otimes \sO$,
defined as \begin{equation}H^0(\sF(-1,0))\otimes \sO(-1)\ni (s,(a,b)) \stackrel{m}{\longmapsto} (b\zeta-a)s\in H^0(\sF)\label{subbundle}\end{equation} (here $(a,b)\in l$, where $l$ is the fibre of $\sO(-1)$ over $[l]$). 
Consequently, all $p_j$ in \eqref{E-split} are equal to $1$.
\par
Thus, we have two maps 
\begin{center} $\Phi$: {\em $\oP^1$ of complex structures $\longmapsto$ sheaf},\\$\Psi$: {\em sheaf $\longmapsto$ $\oP^1$ of complex structures}.\end{center}
We want to compare their compositions. First of all, we have:

\begin{proposition} Let $\sF$ be a $1$-dimensional $\sigma$-sheaf on $\oP^1\times \oP^1$,  such that  $\supp\sF$ does not meet $\ol\Delta$ and $H^\ast(\sF(-1,-1))=0$. Then $\Phi\circ \Psi(\sF)\simeq \sF$ (as $\sigma$-sheaves). 
\end{proposition}
\begin{proof}    A $1$-dimensional sheaf $\sF$, such that $H^\ast(\sF(-1,-1))=0$, is generated by global sections (see, e.g. \cite[Theorem 6.9]{McL-S}). The assumption that $H^0(\sF(-1,-1))=0$ implies also that $\sF(-1,-1)$ (and therefore $\sF$ as well) does not have zero-dimensional subsheaves, i.e. $\sF$ is pure. Hence $\sF$ has  has projective dimension $1$ \cite[\S 1.1]{HL}, and so the surjection $H^0(\sF)\otimes\sO\to \sF$ extends to a resolution by locally free sheaves:
\begin{equation} 0\to \sW\to H^0(\sF)\otimes\sO\to \sF\to 0,\label{F15}\end{equation}
where $\sW$ does not have trivial summands. As $\sF$ is a $\sigma$-sheaf, we can assume that $\sW$ is a $\sigma$-sheaf and the maps are compatible with $\sigma$. Let us write $\sW\simeq \oplus \sO(-p_i,-q_i)$.  Then $p_i,q_i\geq 0$, $p_i+q_i\geq 1$, but, since  $H^\ast(\sF(-1,-1))=0$, either $p_i=0$ or $q_i=0$ for every $i$. We also know, from \eqref{decompose2}, that $h^0(\sF(-1,0))=n$, where $2n=h^0(\sF)$. It follows that $h^1(\sW(-1,0))=h^1(\sW(0,-1))=n$ and, therefore, all nonzero $p_i$ and all nonzero $q_i$ are equal to $1$.  The resolution \eqref{F15} takes the form
\begin{equation} 0\to \sO(-1,0)^{\oplus n}\oplus \sO(0,-1)^{\oplus n}\longrightarrow H^0(\sF)\otimes\sO\longrightarrow \sF\to 0.\label{F16}\end{equation}
The embedding of $\sO(-1,0)^{\oplus n}$ into $H^0(\sF)\otimes\sO$ is given by a matrix of linear polynomials in $\zeta$, and  we  conclude from \eqref{F16} that the fibre of $\sO(-1,0)^{\oplus n}$ at $\zeta$ maps to global sections of $\sF$ vanishing on $D_\zeta=\{\zeta\}\times \oP^1$. Similarly the fibre of $\sO(0,-1)^{\oplus n}$ at $\eta$ maps to global sections of $\sF$ vanishing on $\hat{D}_\eta=\oP^1\times \{\eta\}$.
\par
Now recall the construction of $\Psi$. The sheaf $\sG=\Phi\circ \Psi(\sF)$ fits into the following exact sequence:
\begin{equation} 0\to \sT\stackrel{M}{\longrightarrow} H^0(\sF)\otimes \sO\longrightarrow\sG\to 0, \label{FF}
\end{equation} 
where $\sT=H^0(\sF(-1,0))\otimes \sO(-1,0)\oplus H^0(\sF(0,-1))\otimes \sO(0,-1)$ and the map $M$ is given by $(X,Y)\mapsto m(X)+m(Y)$ with $m$ defined in \eqref{subbundle}. Therefore we have isomorphisms between the first two terms in \eqref{F15} and \eqref{FF}, which make the following diagram commutative:
\begin{equation*}\begin{CD} 0 @>>>\sW @>>> H^0(\sF)\otimes\sO @>>> \sF @>>> 0
\\
@. @VVV  @VVV @. @.\\
0 @>>> \sT @>>> H^0(\sF)\otimes \sO @>>>\sG @>>> 0.
 \end{CD}
\end{equation*}
Hence $\sG\simeq \sF$.
\end{proof}

\medskip

On the other hand, the composition $\Psi\circ \Phi$ does not need to be the identity (even up to isomorphisms). Let us introduce the notion of a morphism of $2$-spheres of complex structures:
\begin{definition} A morphism between $K_1:\oP^1\to \sJ(V_1)$ and $K_2:\oP^1\to \sJ(V_2)$ is a linear map $\alpha:V_1\to V_2$ such that $\alpha\bigl(K_1(\zeta)v\bigr)=K_2(\zeta)(\alpha(v))$ for all $\zeta\in \oP^1$ and $v\in V_1$.\end{definition}
\begin{proposition} Let $K:\oP^1\to \sJ(V)$ be a $\oP^1$ of complex structures. The following conditions are equivalent:
\begin{itemize}
\item[(i)]  $K^\ast \sV^{1,0}\simeq O(-1)^{\oplus n}$.
\item[(ii)] $\Psi\circ \Phi (V,K)$ is isomorphic to $(V,K)$. 
\end{itemize}
\end{proposition}
\begin{proof} If $\Psi\circ \Phi (V,K)$ is isomorphic to $(V,K)$, then clearly (i) holds, due to the above description of the map $\Psi$. Conversely, if (i) holds, then all $p_j$ in \eqref{E-split} are equal to $1$, and the sequence \eqref{F1} implies that we can identify $V^{1,0}_\zeta$ with sections of $\sF$ vanishing on $\{\zeta\}\times \oP^1$. Therefore, by the above discussion, $\Psi$ recovers $(V,K)$ up to an isomorphism.
\end{proof}

\medskip

{\em Remark.} Although we are interested primarily in $2$-spheres of complex structures, all results of this section remain true for holomorphic maps from $\oP^1$ to the full complex Grassmannian  $\Gr_n(V^\cx)$. We only need to remove the assumption or conclusion that $S\cap \ol\Delta=\emptyset$. In particular, we have (here $\sV$ is the tautological bundle on the Grassmannian):
\begin{theorem} There exists a natural 1-1 correspondence between: \begin{itemize}
\item[(a)] holomorphic maps $K:\oP^1\to \Gr_n(\cx^{2n})$ such that $K^\ast\sV=\sO(-1)^{\oplus n}$, modulo the action of $GL_{2n}(\oR)$ on the Grassmannian,\\
{\bf and}
\item[(b)] isomorphism classes of $1$-dimensional $\sigma$-sheaves $\sF$ on $\oP^1\times \oP^1$ such that $\chi(\sF)=2n$ and  $H^\ast(\sF(-1,-1))=0$.
\end{itemize}
\end{theorem}

\section{$O(-1)$-structures and pluricomplex structures on vector spaces\label{pluri}}

The results of the previous section imply that we should expect particularly nice results, if we restrict the class of $2$-spheres of complex structures.

\begin{definition} Let $V$ be a $2n$-dimensional real vector space. An {\em $O(-1)$-structure} on $V$ is a holomorphic map $K:\oP^1\to \sJ(V)$ such that $K^\ast \sV^{1,0}\simeq O(-1)\otimes \cx^n$.\label{definition0}\end{definition}

\begin{remark}
 It follows immediately that $K$ is an embedding.
\end{remark}

From the previous section, we  deduce:
\begin{proposition} There is a 1-1 correspondence between isomorphism classes of $O(-1)$-structures and isomorphism classes of $1$-dimensional $\sigma$-sheaves $\sF$ on $\oP^1\times \oP^1$, such that  $ \ol\Delta\cap \supp\sF=\emptyset$ and $H^\ast(\sF(-1,-1))=0$.\hfill $\Box$\label{FtoK}\end{proposition}
We stress that the isomorphism classes are ``$\sigma$-isomorphism" classes, i.e. we allow only the action of $GL(V)$ on $O(-1)$-structures and we consider sheaves to be isomorphic only if they are isomorphic as $\sigma$-sheaves.

We now consider the sequence \eqref{F0} for a given $O(-1)$-structure. The vector bundle $W$ is isomorphic to $O(-1,0)\otimes \cx^n\oplus O(0,-1)\otimes \cx^n$ and we can rewrite \eqref{F0} as 
\begin{equation}0\to \sO(-1,0)^{\oplus n}\oplus \sO(0,-1)^{\oplus n}\stackrel{M}{\longrightarrow}\sO^{\oplus 2n}\to \sF\to 0. \label{F}\end{equation}
We adopt the following definition.
\begin{definition} The sheaf $\sF$ is called the {\em characteristic sheaf} and its support $S$ the {\em characteristic curve} of an $O(-1)$-structure.
\end{definition}

We recall the description of $S$ given in Proposition \ref{S}.
We observe that the  $\sigma$-invariance implies that $S$ (i.e. the minimal polynomial $p_M$) has bidegree $(k,k)$ for some $k\in \oN$.
\begin{definition} The positive integer $k$ is called the {\em degree} of an $O(-1)$-structure.\end{definition}
\begin{remark} Corollary \ref{VB0} implies that, if the characteristic curve is integral, then its degree $k$ divides $n$ and 
the sheaf $\sF$ is locally free of rank $n/k$. \label{VB1}\end{remark}


We now impose a further condition on $K$.

\begin{definition} Let $V$ be a $2n$-dimensional real vector space. A {\em pluricomplex structure} on $V$ is an $O(-1)$-structure $K:\oP^1\to \sJ(V)$ such that  $V^\cx/K^\ast \sV^{1,0}\simeq \sO(1)\otimes \cx^n$.\label{definition}\end{definition}

In other words, we have a short exact sequence
$$ 0\rightarrow \sO(-1)\otimes \cx^n \to \sO\otimes V^\cx \to \sO(1)\otimes \cx^n\to 0$$
of vector bundles on $\oP^1$. 
Taking the corresponding cohomology sequence, we conclude that a pluricomplex structure is the same as an isomorphism $e:  V^\cx\to \cx^n\otimes \cx^2$, such that for any vector $x\in \cx^2$, $e^{-1}(\cx^n\otimes x)$ is transversal to its complex conjugate.

\begin{remark}
 A generic $O(-1)$-structure is a pluricomplex structure.
\end{remark}

We can characterise pluricomplex structures as follows:
\begin{lemma}
 Let $K:\oP^1\to \sJ(V)$ be an $O(-1)$-structure. The following conditions are equivalent:
\begin{itemize}
 \item[(i)] $K$ is a pluricomplex structure.
\item[(ii)] $V^{1,0}_\zeta\cap V^{1,0}_{\zeta^\prime}=0$ if $\zeta\neq \zeta^\prime$.
\item[(iii)]  $J_\zeta(v)\neq J_{\zeta^\prime}(v)$ for any $\zeta\neq \zeta^\prime$ and $v\in V$.
\item[(iv)] The characteristic sheaf $\sF$ of $K$ satisfies: $H^\ast (\sF(-2,0))=H^\ast (\sF(0,-2))=0$. 
\end{itemize}\label{-2,0}
\end{lemma}
\begin{proof}
(i) is equivalent to (ii), since the existence of an $\sO(r)$-component, $r>1$, in $Q=V^\cx/K^\ast \sV^{1,0}$ is equivalent to the existence of a section of $Q$ vanishing at two points $\zeta,\zeta^\prime$, which, in turn, is equivalent to $V^{1,0}_\zeta\cap V^{1,0}_{\zeta^\prime}\neq 0$. The equivalence of (ii) and (iii) is clear. 
\par
To prove the equivalence of (i) (or (ii)) and (iv) recall from the previous section that  $K$ arises from the sheaf $\sF$.    The space $V^{1,0}_\zeta$ is identified with sections of $\sF$ vanishing on $\{\zeta\}\times\oP^1$ and, consequently (ii) is equivalent to $h^0(\sF(-2,0)=0$. Tensoring   the sequence \eqref{F0} with $\sO(-2,0)$, and taking the long exact cohomology sequence, shows (as all $p_j=-1$) that the vanishing of $H^0 (\sF(-2,0))$ is equivalent to the vanishing of $H^1 (\sF(-2,0))$.
Finally, using the $\sigma$-invariance, we conclude that $h^1(\sF(0,-2))=h^0(\sF(0,-2))=0$. 
\end{proof}

We have, therefore, the following classification result. 
\begin{proposition} The isomorphism classes of pluricomplex structures on a real vector space $V$ are in 1-1 correspondence with isomorphism classes of  $\sigma$-sheaves $\sF$ on $\oP^1\times \oP^1-\ol \Delta$, supported on a  compact $1$-dimensional complex subspace, and such that 
\begin{itemize}
\item[(i)] $\chi (\sF)=\dim V$;
\item[(ii)] the cohomology of $\sF(-1,-1), \sF(-2,0), \sF(0,-2)$ vanishes. 
\end{itemize}\label{class}
\end{proposition}
\begin{proof} This follows immediately from Proposition \ref{FtoK}  and from Lemma \ref{-2,0}.
\end{proof}

\begin{remark} As we show in the Appendix, such a sheaf automatically satisfies $H^\ast(\sF(p,q))=0$ for all $p,q\in \oZ$ with $p+q=-2$. This seems to reflect an interesting property of the $\sigma$-invariant part $\Theta^\sigma$ of the theta divisor of a characteristic curve $S$. 
\end{remark}


Let us give a matrix description of isomorphism classes of pluricomplex structures on $V$.
\begin{proposition} Let $V$ be a $2n$-dimensional vector space. The set of isomorphism classes of pluricomplex structures on $V$ is in bijection with
$$ \left\{(X,Y)\in \gl(n,\oC)\times\gl(n,\oC);\;\forall_{\zeta\in\oP^1} \forall_v \enskip(X+\zeta Y)v\neq \pm \bar{\zeta}\bar{v}\right\}/GL(n,\cx),$$
where the action of $GL(n,\cx)$ is given by $g.(X,Y)=\bigl(gX\bar g^{-1}, gY\bar g^{-1}\bigr)$.
\label{matrices}\end{proposition}
\begin{remark} The condition on $(X,Y)$ is equivalent to
$$ \det\bigl( (X+\zeta Y)(\bar X+\bar\zeta\bar Y)-|\zeta|^2 I\bigr)\neq 0, \enskip \forall_{\zeta\in\oP^1}.$$
\end{remark}
\begin{proof} A pluricomplex structure is given by a map $O(-1)\otimes \cx^n \to O\otimes V^\cx$. Such a map must be of the form $A+B\zeta$, where $A$ and $B$ are linear maps from $\cx^n$ to $V^\cx$.
Let us fix the complex structure $J_0$ and decompose $V^\cx$ into $V^{1,0}\oplus V^{0,1}$ for this complex structure. Thus $V^\cx=\cx^n\oplus \cx^n$, the map $A$ becomes $Av=(v,0)$ and complex conjugation is replaced by the real structure $\tau(v,w)=(\bar w,\bar v)$. Let $B=(B_1,B_2)$. Then $V^{1,0}_\zeta=\{(v+\zeta B_1v,\zeta B_2 v);\;v\in \cx^n\}$ and since $V^{1,0}_\zeta\cap V^{1,0}_{\zeta^\prime} =0$ for $\zeta\neq \zeta^\prime$, we conclude that $B_2$ is invertible. Thus we can write $V^{1,0}_\zeta=\{(Xv+\zeta Yv,\zeta  v);\;v\in \cx^n\}$, where $X=B_2^{-1}$ and $Y=B_1B_2^{-1}$. The condition that $\tau\bigl(V^{1,0}_\zeta\big)\cap V^{1,0}_\zeta=0$ translates into the one on $(X,Y)$ given in the statement. We also observe that the invertibility of $X$ implies that $V^{1,0}_\zeta\cap V^{1,0}_{\zeta^\prime} =0$ for $\zeta\neq \zeta^\prime$. Finally, the $GL(n,\cx)\subset GL(2n,\oR)$ fixing $J_0$ acts on $V^\cx$ by $g.(v,w)=(gv,\bar gw)$, and its induced action on $\{(X,Y)\}$ is as in the statement.\end{proof}

\begin{remark}  The proof of the above proposition shows that the whole pluricomplex structure is determined by $3$ complex structures $J_0,J_\infty$ and $J_1$. Also, the set of triples which determine a pluricomplex structure is open. We shall, however,  see shortly that there are no pluricomplex structures for odd $n$.\end{remark}

\begin{remark} If  a pluricomplex structure is given by a pair of matrices $X,Y$ as in Proposition \ref{matrices}, then the map $M(\zeta,\eta):\sO(-1,0)\otimes \cx^n\oplus \sO(0,-1)\otimes \cx^n\to \sO\otimes \cx^{2n}$ defining $\sF$ in \eqref{F} is given by 
$$ M(\zeta,\eta)=\begin{pmatrix} X+\zeta Y & -1\\ \zeta & \eta\ol X-\ol Y\end{pmatrix},$$
and the set-theoretic support of $S$ is 
$$\{(\zeta,\eta);\; \det M(\zeta,\eta)=0\}=\{(\zeta,\eta);\; \det\bigl((\eta\ol X-\ol Y)(X+\zeta Y)+\zeta I\bigr)=0\}.$$
\label{M}\end{remark}

\begin{proposition} If a $2n$-dimensional vector space admits a pluricomplex structure, then $n$ is even.\label{even}
\end{proposition} 
\begin{proof} It follows from Remark \ref{M} and  Proposition \ref{matrices}  that the set of matrices defining $(S,\sF)$  is open and so we can assume that $S$ is smooth and $k=n$. Corollary \ref{VB0} implies that $\sF$ is a line bundle.  Since $\sF$ is a $\sigma$-sheaf and $\sigma$ does not have fixed points on $\supp \sF$,  the degree of $\sF$  is even. On the other hand 
$$\deg(\sF)=\chi(\sF)+g(S)-1=2n+(k-1)^2-1=n^2.$$
Therefore $n$ is even. \end{proof}

On a $4m$-dimensional vector space there are many pluricomplex structures. The well-known ones are hypercomplex structures. They satisfy $J_{-1/\bar\zeta}=-J_{\zeta}$, for all $\zeta\in \oP^1$. This property characterises hypercomplex structures among all pluricomplex structures.
\begin{proposition} A pluricomplex structure is hypercomplex if and only if its characteristic curve is the diagonal $\{(\zeta,\zeta)\}\subset \oP^1\times\oP^1$.\label{HC}\end{proposition}
\begin{proof} It is clear from the above remark that the characteristic curve of a hypercomplex structure is the diagonal. Conversely, if $(S,\sO_S)$ is the diagonal, then it has bidegree $(1,1)$ and is integral. Proposition \ref{VB0} implies that $\sF$ has rank $n$. If we now use the matrix description of $M$ and $S$ given in Remark \ref{M}, then the fact that $(0,0)\in S$ implies that $\Ker \ol Y X$ is $n$-dimensional. Since $X$ is invertible, $Y=0$. Repeating the argument for $(1,1)\in S$ shows that $\Ker (\ol X X+I)$ is $n$-dimensional and, hence $ \ol X X=-I$. The proof of Proposition \ref{matrices} shows now that the pluricomplex structure is hypercomplex.
\end{proof}

In particular, a hypercomplex structure is a pluricomplex structure of degree $1$. As the next remark shows, pluricomplex structures of degree $1$ are never far from being hypercomplex (see also Example \ref{examples}(ii)).
\begin{remark} A $\sigma$-invariant curve of bidegree $(1,1)$ can be transformed into the diagonal by an element of $PSL(2,\cx)$ acting on $\oP^1\times \oP^1$ via $g.(\zeta,\eta)=\bigl(g.\zeta,(g^\ast)^{-1}.\eta\bigr)$.
Therefore, if $K:\oP^1\to \sJ(V)$ is a pluricomplex structure of degree $1$, then there exists a $g\in PSL(2,\cx)$ such that $K\circ g$ is a linear hypercomplex structure. In particular,  for every $\zeta$ there is a $\zeta^\prime$ such that $J_{\zeta^\prime}$ and $J_\zeta$
anti-commute. \label{degree=1}\end{remark}

\section{A hypercomplex extension of a pluricomplex structure\label{hyperc}}

In this section we shall connect arbitrary pluricomplex structures to hypercomplex structures, but defined on another vector space.
\par
Let $V$ be a $2n$-dimensional real vector space equipped with a pluricomplex structure $K:\oP^1\to \sJ(V)$. Recall that $\wh K=-K\circ \sigma$, and $V^{1,0}_\zeta, V^{0,1}_\zeta$ (resp. $\wh V^{1,0}_\zeta, \wh V^{0,1}_\zeta$) denote spaces of vectors of type $(1,0)$ and $(0,1)$ for $J_\zeta=K(\zeta)$ (resp. $\wh J_\zeta=\wh K(\zeta)$). By definition, $V^{0,1}_\zeta=\wh V^{1,0}_{-1/\bar\zeta}$. 
\par
We consider the vector space $V^\cx\oplus V^\cx$ equipped with the real structure 
\begin{equation}\tau(v,w)=(\bar w,\bar v),\label{tauV}\end{equation}
the fixed-point set of which is 
\begin{equation} V_\Delta=\{(v,\bar{v})\}\subset V^\cx\oplus V^\cx.\label{V_Delta}\end{equation}
We define a hypercomplex structure 
 $K_\Delta:\oP^1\to \sJ(V_\Delta)$ on $V_\Delta$ by specifying spaces $(V_\Delta)^{1,0}_\zeta, (V_\Delta)^{0,1}_\zeta$ of vectors of type $(1,0)$ and $(0,1)$ in $V_\Delta^\cx=V^\cx\oplus V^\cx$:
\begin{equation} (V_\Delta)^{1,0}_\zeta=V^{1,0}_\zeta\oplus \wh V^{1,0}_\zeta,\quad  (V_\Delta)^{0,1}_\zeta=V^{1,0}_{-1/\bar\zeta}\oplus \wh V^{1,0}_{-1/\bar\zeta}.\label{HCC}\end{equation}

\begin{lemma} Formulae \eqref{HCC} define a hypercomplex structure on $V_\Delta$.\label{hc}
\end{lemma}
\begin{proof} Since $K$ is a pluricomplex structure, $(V_\Delta)^{1,0}_\zeta$ and $(V_\Delta)^{0,1}_\zeta$ have trivial intersection for each $\zeta$. It is clear that the involution $\tau$ exchanges $(V_\Delta)^{1,0}_\zeta$ and $(V_\Delta)^{0,1}_\zeta$ and, so, we obtain a holomorphic embedding $K_\Delta: \oP^1\to \sJ(V_\Delta)$. The spaces $(V_\Delta)^{1,0}_\zeta$ form a bundle on $\oP^1$ isomorphic to a direct sum of $\sO(-1)$'s, since both $V^{1,0}_\zeta$ and $\wh V^{1,0}_\zeta$ do. Furthermore, the spaces $(V_\Delta)^{1,0}_\zeta$ have trivial intersection for two different $\zeta$, since $K$ is a pluricomplex structure. Therefore $K_\Delta$ is also a pluricomplex structure. Finally, observe that $(V_\Delta)^{0,1}_\zeta=(V_\Delta)^{1,0}_{-1/\bar\zeta}$ and, hence, Proposition \ref{HC} implies that $K_\Delta$ is a hypercomplex structure.\end{proof}

\medskip

In addition to the hypercomplex structure \eqref{HCC}, which is determined by a given pluricomplex structure, $V_\Delta$ has also a natural complex structure $I(v,\bar v)=(iv,-i\bar v)$. The spaces of $(1,0)$- and $(0,1)$-vectors for $I$ are therefore
\begin{equation} (V_\Delta)_I^{1,0}=V^\cx\oplus 0,\quad (V_\Delta)_I^{0,1}=0\oplus V^\cx.\label{I}
\end{equation}
It is clear that the hypercomplex structure $K_\Delta$ and the complex structure $I$ commute. In other words,
the algebraic structure of $V_\Delta$ is that of a {\em biquaternionic} module, i.e. a module over the ring $M_{2\times 2}(\cx)$ of complex $2\times 2$ matrices.
\par
We now break the symmetry, and identify $V_\Delta$ with the left copy of $V^\cx$.  Let us write $\tilde K$ for $K_\Delta$ and $\tilde J_\zeta$ for $\tilde{K}_\zeta$ acting on this $V^\cx$. We restate \eqref{HCC} and describe the hypercomplex structure $\tilde{K}$ as follows.
\par
For $\zeta\in \oP^1$ we have
\begin{equation} V^\cx=V^{1,0}_\zeta\oplus V^{1,0}_{-1/\bar\zeta}.\label{decomp}\end{equation}
Decompose any $v\in V^\cx$ as $v=v_1+v_2$ with $v_1\in V^{1,0}_\zeta$, $v_2\in V^{1,0}_{-1/\bar\zeta}$. Then
\begin{equation} \tilde J_\zeta(v)=i(v_1-v_2).\label{tilde J}\end{equation}

We now have
\begin{theorem} There is a natural 1-1 correspondence between \begin{itemize}
\item[(i)] pluricomplex structures on $V$ and 
\item[(ii)] hypercomplex structures $\tilde K:\oP^1\to \sJ(V^\cx)$ on $V^\cx$, which commute with the intrinsic complex structure $I$ and satisfy
\begin{equation} V\cap \{X-i\tilde J_\zeta X;\: X\in V^\cx\}=0,\enskip \forall \zeta\in \oP^1.\label{condition}\end{equation}
\end{itemize}\label{amazing}\end{theorem}
Condition \eqref{condition} means that there no real eigenvectors for any $\tilde J_\zeta$.
\begin{proof}
 Equations \eqref{decomp} and \eqref{tilde J} describe the passage from (i) to (ii). Conversely, suppose that we have a hypercomplex structure  as in (ii). We set 
$$ V^{1,0}_\zeta= \{X-i\tilde J_\zeta X;\: X\in V^\cx\}.$$
Condition \eqref{condition} implies that this defines a complex structure on $V$. Since $\tilde K$ is a hypercomplex structure, these subspaces form a bundle  over $\oP^1$, which is a sum of $O(-1)$'s, and they have zero intersection for two different values of $\zeta$. Therefore we obtained a pluricomplex structure, and it is clear that the two constructions are inverse to each other and functorial.
\end{proof}


\section{Pluricomplex manifolds\label{plurimflds}}

Let $M$ be a smooth manifold and $E$ a vector bundle of even rank. We can associate to $E$ the twistor bundle $\sJ(E)$ with fibre $\sJ(E)_m=\sJ(E_m)$. If $M$ is even-dimensional, we write $\sJ(M)$ for $\sJ(TM)$.

\begin{definition} A pluricomplex structure on a vector bundle $E$ is a smooth map $K:M\times \oP^1\to \sJ(E)$ such that $K(m,\cdot)$ is a linear pluricomplex structure on $E_m$ for each $m\in M$. The bundle $E$ is then called a {\em pluricomplex bundle}. \label{J(E)}\end{definition}

\begin{definition} Let $M$ be smooth manifold of dimension $4m$. A pluricomplex structure on $M$ is a pluricomplex structure on $TM$.\label{J(M)}\end{definition}

\begin{remark} The structure group of a generic pluricomplex structure is trivial, so such a manifold is parallelisable.\end{remark}
\begin{definition} A pluricomplex structure on $M$ is said to be {\em integrable} if every $J_\zeta=K(\cdot,\zeta)$ is  integrable. A {\em pluricomplex manifold} is a manifold equipped with a integrable pluricomplex structure.\end{definition}

\begin{remark} We can define integrable $O(-1)$-structures on manifolds in exactly the same way.\end{remark}

We now transfer to manifolds the correspondence between pluricomplex and hypercomplex structures described in Theorem \ref{amazing}. Thus, a pluricomplex structure on a vector bundle $E$ corresponds to a biquaternionic structure on its complexification $E^\cx$. In other words, we have a hypercomplex structure on $E^\cx$ commuting with the tautological complex structure.
\par
Let $M$ be a pluricomplex manifold. Since $M$ is real-analytic it has a complex thickening $M^\cx$, equipped with an anti-holomorphic involution $\tau$, the fixed-point set of which is $M$. The hypercomplex structure on $T^\cx M$ extends to an integrable hypercomplex structure on $M^\cx$. Let us adopt the following definition
\begin{definition}
 A {\em bihypercomplex manifold} is a complex manifold equipped with an integrable hypercomplex structure such that the complex and the hypercomplex structures commute.
\end{definition}
\begin{remark} Bihypercomplex manifolds are examples of GS-manifolds of Manin \cite{M} or paraconformal manifolds of Bailey and Eastwood \cite{Ba-E}. See also \cite{JV} for another point of view on these objects.
\end{remark}
The results of the previous section imply:
\begin{proposition} 
 Any pluricomplex manifold $M$ arises as fixed-point set of an involution $\tau$ on a bihypercomplex manifold $(\tilde M, I, J_\zeta)$, such that 
\begin{itemize} 
 \item[(i)] $\tau$ is antiholomorphic with respect to the complex structure $I$, and
\item[(ii)] for any  $\zeta\in \oP^1$, $J_\zeta$ has no eigenvectors in $TM$. 
\end{itemize}
\end{proposition}

As is well known, any hypercomplex manifold has a canonical torsion free connection, called the Obata connection. On a bihypercomplex manifold the complex structure $I$ is automatically parallel for the Obata connection:
\begin{lemma}
 Let $(\tilde M, I, J_\zeta)$ be a bihypercomplex manifold. There exists a unique torsion-free linear connection $\tilde\nabla$ on $\tilde M$ such that both  the complex structure $I$ and the hypercomplex structure $\{J_\zeta\}$ are $\tilde\nabla$-parallel.
\end{lemma}
\begin{proof} One imitates the proof of existence and uniqueness of the Obata connection (see, e.g. \cite{A-M}). \end{proof} 


The above discussion proves:
\begin{theorem} Let $M$ be a pluricomplex manifold of dimension $2n$. Then 
 $T^\cx M$ is equipped with a canonical linear connection $\tilde\nabla$ such that
\begin{itemize}
\item[(i)] Both the hypercomplex structure $\tilde K$ and the tautological complex structure $I$ on $T^\cx M$ are $\tilde\nabla$-parallel.
\item[(ii)] Consequently, there is a complex bundle $E$ of rank $n$ on $M$, such that $T^\cx M\simeq E\otimes {\cx^2}$, and the holonomy of $\tilde\nabla$ is a subgroup of $GL(E_m)\simeq GL(n,\cx)$.
\item[(iii)] The connection $\tilde\nabla$ is the restriction of a torsion-free connection on a complex thickening $M^\cx$ with its canonical integrable bihypercomplex structure. 
\end{itemize}\label{nabla}
\end{theorem}

\medskip

We shall now define a canonical connection on a pluricomplex manifold. Consider the bundle maps
$$ i:TM\to T^\cx M,\quad p:T^\cx M\to TM$$
defined by $i(v)=v$, $p(v)=\re v$.
\begin{definition} Let $M$ be a pluricomplex manifold. The canonical connection $\nabla$ of the pluricomplex structure is defined by
 $$\nabla_X Y=p\bigl( \tilde\nabla_{i(X)} i(Y)\bigr).$$
\end{definition} 

The following is obvious from the definition and Theorem \ref{nabla}(iii):
\begin{proposition} The canonical connection of an integrable pluricomplex structure is torsion-free.
\hfill $\Box$\label{torsion-free}
\end{proposition}

\bigskip

There is a class of natural metrics on a pluricomplex manifold. 
\begin{definition} Let $M$ be a pluricomplex manifold. A Riemannian metric $g$ on $M$ is said to be {\em plurihermitian} if it is invariant with respect to each complex structure $J_\zeta$, $\zeta\in \oP^1$.
\end{definition}

On the other hand, an analogous definition of  a {\em plurik\"ahlerian} metric does not seem to be very interesting: since such a definition leads to reduction of the holonomy group, we expect that such a metric (at least on simply connected manifolds) is a product of a hyperk\"ahler metric and a flat metric. Instead, we adopt the following definition:

\begin{definition} Let $M$ be a pluricomplex manifold. A Riemannian metric $g$ on $M$ is said to be {\em plurik\"ahlerian} if it is a restriction to $TM$ of an (indefinite) metric $\tilde{g}$ on $T^\cx M$ which satisfies the following two conditions:
\begin{itemize} 
\item[(i)] $\tilde g$ is hyperhermitian for the hypercomplex structure on  $T^\cx M$ and satisfies $\tilde g(Iv,Iw)=-\tilde g(v,w)$  for the tautological complex structure $I$.
\item[(ii)] $\tilde g$ is 
 parallel for the canonical connection $\tilde\nabla$.
\end{itemize}
\end{definition} 
\begin{remark} $\tilde g$ must have signature $(\dim M,\dim M)$.\end{remark}
\begin{remark} A plurik\"ahlerian metric is plurihermitian.
\end{remark}

A natural way of obtaining such a metric is from $(2,0)$-forms along fibres of the twistor space of $M^\cx$ (which is a hypercomplex manifold), similarly to the hyperk\"ahler case.  Suppose that we have a
section $\Omega$ of $\Lambda^2\bigl(T^{1,0} M)^\ast$ on $M\times \oP^1$ (here $T^{1,0} M$ at $(m,\zeta)$ are the tangent vectors at $m$ of type $(1,0)$ for $J_\zeta$), which defines a holomorphic symplectic form for each $\zeta\in \oP^1$.
Thus, $\Omega$ is a holomorphic choice of a nondegenerate closed $(2,0)$-form for every $J_{\zeta}$. Since $K^\ast\sV^{1,0}\simeq \sO(-1)\otimes \cx^{n}$, $\Omega$ is a section of $\sO(2)$ and we can write
\begin{equation} \Omega=\omega_0+\zeta\omega_1+\zeta^2\omega_2,
\end{equation}
for some closed complex-valued $2$-forms $\omega_0,\omega_1,\omega_2$ on $T^\cx M$. Alternatively, $\Omega$  is an ordinary symplectic form $\omega$ on $E$ defined in Theorem \ref{nabla}(ii). As in the hypercomplex case, we can define a $\cx$-valued metric on $T^\cx M$:
$$ g^\cx(a+b\zeta,a+b\zeta)=\omega(a,b).$$
 Let $\tilde g$ be the real part of $g^\cx$. By construction $\tilde g$ is parallel for the canonical connection $\tilde\nabla$. Condition (i) of the above definition is also satisfied, so we obtain a plurik\"ahlerian metric as soon as  its restriction to $TM$ is positive-definite. In such a case, we have:
\begin{proposition} The Levi-Civita connection of a plurik\"ahlerian metric obtained as above coincides with the canonical connection of the pluricomplex structure if and only if the restriction of $g^\cx$ to $TM$ is real.\hfill $\Box$
\end{proposition} 

\begin{remark} We can use bihypercomplex extensions to define pluricomplex quotients. Let $M$ be a pluricomplex manifold and $G$ a Lie group acting on $M$ preserving the pluricomplex structure. The action of $G$ extends to a (local) action of $G^\cx$ on $M^\cx$ preserving the hypercomplex structure $\tilde{K}$ and the tautological complex structure $I$. Suppose that there exists a hypercomplex moment map \cite{Joyce} $\tilde\mu: M^\cx\to \fG^\cx\otimes \oR^3$ for the action of $G^\cx$. The zero-level set $\tilde\mu^{-1}(0)$ is invariant under $G^\cx$. We make an additional assumption about $\tilde\mu$: we suppose that $\tilde\mu^{-1}(0)$ is invariant under the involution $\tau$. Then $\bigl(\tilde\mu^{-1}(0)\bigr)^\tau/G$, if a manifold, will have a natural integrable pluricomplex structure and can be regarded as a pluricomplex quotient of $M$ by $G$. Its bihypercomplex extension is an open subset of $\tilde\mu^{-1}(0)/G^\cx$ (the latter quotient simply means identifying points lying on the same local orbit of $G^\cx$  in a neighbourhood of $\bigl(\tilde\mu^{-1}(0)\bigr)^\tau$).
\end{remark}

\begin{example}[Compact examples] Let $K:\oP^1\to \sJ(V)$ be a linear pluricomplex structure with characteristic sheaf $\sF$. Quotienting by a lattice gives an integrable pluricomplex structure on a $4n$-dimensional torus ($4n=\dim V$). Another class of examples is produced by  quotienting $V-\{0\}$ by the group of dilatations generated by a $\lambda>0$. The resulting Hopf manifold is diffeomorphic to $S^{4n-1}\times S^1$, and it carries the induced integrable pluricomplex structure. The characteristic sheaf  at every point of $M_\lambda$ is $\sF$. See example \ref{examples} for more on this.\label{Hopf}
\end{example}

\begin{example}[Plurihermitian structure of $H^3\times \oR$]
The product of the hyperbolic $3$-space and $\oR$ has a natural  integrable pluricomplex structure and the product metric is plurihermitian. The complex structures $J_\zeta$ are parameterised by points on the ideal boundary $S^2\simeq \oP^1$. Once we choose such a point, $H^3$ can be viewed as the upper half-space  $\{(x,y,z); z>0\}$ with metric $\frac{1}{z^2}(dx^2+dy^2+dz^2)$. We change coordinates to $(x,y,t)\in \oR^3$, with $z=e^{-t}$, so that $t$ is the unit speed parameter on the geodesic $(x,y, e^{-t})$. The metric becomes $e^{2t}(dx^2+dy^2)+dt^2$. The metric on $H^3\times S^1$ is
$$ e^{2t}(dx^2+dy^2)+dt^2+ds^2.$$ The complex structure is then defined by
\begin{equation} J_\zeta(dx)=-dy,\quad J_\zeta(dt)=-ds.\label{j}\end{equation}
It is clear that these complex structures are all integrable (each one giving an identification $H^3\times \oR\simeq \cx^2$), they preserve the product metric and that the action of $PSL(2,\cx)\times \oR$ preserves the $2$-sphere of complex structures (and acts on it via the induced action on $\oP^1$). Finally, it is not hard to see that for each point of $H^3\times \oR$, the $2$-sphere of complex structures contains $3$ anticommuting ones (for an $(x,s)\in H^3\times \oR$, let $G\simeq SO(3,\oR)\subset PSL(2,\cx)$ stabilise $x$; then any $3$ complex structures corresponding to an orthonormal basis of $\fG\simeq \so(3,\oR)$ will anticommute).  This means that this $\oP^1$ of complex structures is a pluricomplex structure. In fact, it is a special type of a quaternionic structure, as considered in Remark \ref{degree=1}.\label{basic}
\end{example}

\section{Strongly integrable $O(-1)$- and pluricomplex structures\label{strong}}

Let $M$ be a pluricomplex manifold, i.e. $TM$ is equipped with a pluricomplex structure $K:M\times \oP^1\to \sJ(M)$ such that each complex structure $K(\cdot,\zeta)$ is integrable. Since its complex thickening $M^\cx$ has an integrable bihypercomplex structure, there is a corresponding twistor space, which allows to recover $M^\cx$ as the space of rational curves. We cannot, however, recover $M$ itself.  We now proceed to define a stronger notion of integrability, which allows to obtain $M$ as a space of curves in a  twistor space.
\par
For a pluricomplex structure on a manifold $M$, consider the pointwise sequence \eqref{F}, i.e.
\begin{equation} 0\to\sW_m\to \sO\otimes T_m^\cx M\to \sF_m\to 0, \label{Fm}\end{equation}
for every $m\in M$. We denote the support of $\sF_m$ by $S_m$. 
We consider the following fibration over $M$
\begin{equation} Y=\{(m,\zeta,\eta)\in M\times\oP^1\times\oP^1;\enskip (\zeta,\eta)\in S_m\}\stackrel{\nu}\longrightarrow M ,\label{Y}\end{equation}
 and the corresponding map $p:Y\to  \oP^1\times\oP^1$. We shall assume that $Y$ is smooth (it is certainly smooth whenever  $S_m$ is smooth) and that $p$ is a submersion onto its image. 
\begin{definition} A pluricomplex structure is said to be {\em strongly integrable} if 
$T^{1,0}_\zeta+\wh T^{1,0}_\eta$ is a subbundle of  $T^\cx p^{-1}(\zeta,\eta)$ for all $(\zeta,\eta)\in \im p;$ and the resulting distribution $\sD\subset T^\cx Y$ is involutive (i.e. $[\sD,\sD]\subset \sD$). A manifold equipped with a strongly integrable pluricomplex structure will be called {\em strongly pluricomplex}.\end{definition}


If the pluricomplex structure has degree $r$ and $\dim M=2rk$, then $\sD$ has corank $r$.
\par

The general theory of involutive structures tells us that $\sD^\oR=\sD\cap \overline{\sD}$ is a real integrable distribution on $Y$ and that the space of leaves $Y/\sD^\oR$ is holomorphic. 
\begin{definition} The twistor space $Z$ of a strongly integrable pluricomplex structure is the space of leaves $Y/\sD^\oR$.
\end{definition} 
\begin{remark} $Z$ is a manifold only if the foliation $\sD^\oR$ is simple. In general, one has to view $Z$ as a foliation. \end{remark}

We shall summarise the properties of $Z$. Recall that $\sigma:\oP^1\to \oP^1$ denotes the antipodal map and we have an antiholomorphic involution  $\sigma:\oP^1\times \oP^1\to \oP^1\times \oP^1$, $\sigma(a,b)=\bigl(\sigma(b),\sigma(a)\bigr)$. Its fixed point set is the antidiagonal $\ol\Delta=\{(z,\sigma(z))\}$.

\begin{proposition} (i) $Z$ is equipped with
 a holomorphic map $\rho:Z\to \oP^1\times \oP^1- \ol\Delta$ and an antiholomorphic involution $\tau:Z\to Z$ such that $\rho\circ\tau=\sigma\circ\rho$.\newline
(ii) For every $m\in M$, the curve $\{m\}\times S_m\subset Y$ descends to a holomorphic curve $\hat S_m\subset Z$, such that $\rho_{|\hat S_m}$ is an isomorphism and the normal sheaf $N_m$ of $\hat S_m$ is canonically isomorphic to ${\sF_m}$.
\end{proposition}

\begin{remark} We can define the notion of a strongly integrable $O(-1)$-structure and its twistor space in exactly the same way. The above proposition remains true.  See section \ref{zero-mass} for an example of a manifold with a natural strongly integrable $O(-1)$-structure, which is not pluricomplex anywhere.\label{sO(-1)}
\end{remark}

Before stating the converse, let us recall an appropriate result from deformation theory \cite{Pal}. Let  $Z$ be a complex space and $X\subset Z$ a closed complex space. If $X$ is locally complete intersection and the normal sheaf $\sN_{X/Z}$ of $X$ satisfies $H^1(X, \sN_{X/Z})=0$, then the Douady space $H_Z$, parameterising closed complex subspaces of $Z$, is smooth at $X$.  The dimension of the tangent space $T_XH_Z$ is then equal to $h^0(X, \sN_{X/Z})$. Furthermore, if $Z$ is smooth, then the normal sheaf of a locally complete intersection is locally free.

We have:
\begin{theorem} Let $Z$ be a complex manifold equipped with  a holomorphic map $\rho:Z\to \oP^1\times \oP^1 - \ol\Delta$ and an antiholomorphic involution $\tau:Z\to Z$ such that $\rho\circ\tau=\sigma\circ\rho$. Let $\sC_0$ be a maximal connected subset of the Douady space $H_Z$, consisting of complex subspaces  $S$ satisfying the following conditions:
\begin{itemize} 
\item[(i)] 
$\rho_{|S}$ is a biholomorphism onto its image $\rho(S)$, which is a compact divisor of $\sO(k,k)$ for some $k\in \oN$;
\item [(ii)] the sheaf $\sF=(\rho_{|S}^{-1})^\ast \sN_{S/Z}$ satisfies
$H^\ast\bigl(\rho(S), \sF(-1,-1)\bigr)=0.$
\end{itemize}
Let $\sC_1$ be the maximal open subset of $\sC_0$, such that the $S\in \sC_1$ satisfy, in addition:
\begin{equation*} H^\ast\bigl(\rho(S), \sF(-2,0)\bigr)=H^\ast\bigl(\rho(S), \sF(0,-2)\bigr)=0. \end{equation*}
 Then $\sC_0$ is a smooth manifold and its $\tau$-invariant part $M_0=(\sC_0)^\tau$ (resp. the $\tau$-invariant part $M_1=(\sC_1)^\tau$ of  $\sC_1$)  is equipped with a natural strongly integrable and integrable $O(-1)$-structure (resp.   a natural strongly integrable and integrable pluricomplex structure). Moreover the twistor spaces of $M_0$ and $M_1$ are the corresponding open subsets of $Z$.\label{twistor}\end{theorem}
\begin{proof} First of all, it follows from (i) that $S$ is locally complete intersection (this is an intrinsic condition). Moreover,  (ii) implies that $H^1(\rho(S), \sF)=0$. Therefore
$H^1(S, \sN_{S/Z})=0$ for any $S\in \sC_0$, and $H_Z$ is smooth at every point of $\sC_0$, according to  the remark before the statement of the theorem. Since the conditions (i) and (ii) are open, $\sC_0$ is an open subset of the smooth locus of $H_Z$, hence a manifold. 
\par
Since $Z$ is smooth, the normal bundle $\sN_{S/Z}$ is locally free. 
An $O(-1)$- or pluricomplex structure arises as in Section \ref{spheres}. Writing $N$ for $\sN_{S/Z}$, we have  $$\bigl(T^{1,0}_\zeta\bigr)_m M_0\simeq H^0(S,N(-1,0))\otimes \zeta\subset  H^0(S,N)\simeq T^\cx_m M_0,$$ 
for each $\zeta\in \oP^1$. We claim that these complex structures are integrable (so that the $O(-1)$-structure is integrable). Indeed, $H^0(S,N(-1,0))\otimes \zeta\subset  H^0(S,N)$ corresponds to sections of $N$ vanishing at $\rho^{-1}\bigl(\{\zeta\}\times \oP^1\bigr)\cap S$. This is the tangent bundle to the submanifold of $(M_0)^\cx=\sC_0$ consisting of all curves passing  through $\rho^{-1}\bigl(\{\zeta\}\times \oP^1\bigr)\cap S$. Here, a closed complex subspace $A$ {\em passes through} a closed complex subspace $B$, if $A\cap B=B$ as complex spaces. Therefore $T^{1,0}_\zeta$ is involutive.
\end{proof}

\begin{corollary} A strongly integrable pluricomplex structure is integrable.\hfill $\Box$\label{s-i}\end{corollary}
\begin{example} (i) An integrable hypercomplex structure is strongly integrable as a pluricomplex structure. Its twistor space $Z$ is the usual twistor space fibreing over $\oP^1$ and $\rho$ is the composition of the projection onto $\oP^1$ followed by the diagonal embedding into $\oP^1\times \oP^1$.
\par
(ii) Let $M$ be equipped with a strongly  integrable pluricomplex structure of degree $1$. Then $M$ is the parameter space of $\tau$-invariant rational curves with normal bundle isomorphic to a direct sum of $O(1)$'s. Therefore $M$ is a quaternionic manifold. At every point $m\in M$ we have a $\oP^1$ of almost complex structures, which behave algebraically as the $\oP^1$ of unit imaginary quaternions; these arise as in Remark \ref{degree=1}.
The twistor spaces of the pluricomplex structure and of the quaternionic structure coincide. If we view $Z$ as the twistor space of the quaternionic structure, then $Z$ is equipped with a projection onto $M$, the fibres of which are the $\tau$-invariant $\oP^1$-s. Having a pluricomplex structure as well means that this fibration is trivial, $Z\simeq M\times \oP^1$, and the projection $\pi:Z\to \oP^1$ is holomorphic (but $\pi\circ \tau\neq \sigma\circ\pi$, unless $M$ is hypercomplex). The map $\rho:Z\to \oP^1\times \oP^1 - \ol\Delta$ is given by $\rho(z)=\bigl(\pi(z),\sigma\circ\pi\circ\tau(z)\bigr)$.
\par
(iii) The pluricomplex structures on $S^{4n-1}\times S^1$ described in Example \ref{Hopf} are strongly integrable (at least when $\sF$ is locally free on its support $S$). The twistor space $Z_\lambda$ is the total space of the projective bundle on $S$, obtained by quotienting the total space of the vector bundle $\sF$ minus the zero section by the group generated by fibrewise multiplication by $\lambda$. The map $\rho$ is the projection to $S$, followed by the embedding into $\oP^1\times \oP^1-\ol\Delta$. Perhaps other interesting (compact?) examples can be obtained by taking fibrewise Hilbert schemes of points on $Z_\lambda\to S$. 
\par
(iv) The pluricomplex structure on $H^3\times \oR$, described in Example \ref{basic}  is strongly integrable. The twistor space is the total space of a rank $1$ affine bundle over  $\oP^1\times \oP^1 - \ol\Delta$. If we consider instead $H^3\times S^1$, then the twistor space of the pluricomplex structure can be described as the total space of the line bundle $O(1,-1)$ on $\oP^1\times \oP^1 - \ol\Delta$ with the zero section removed. See also section \ref{instantons} for more about this example.
\par
(v) As we shall show in section \ref{hyperbolic}, the moduli space $\sM_{k,m}$ of framed hyperbolic monopoles of charge $k$ and half-integer mass $m$ has a natural strongly integrable pluricomplex structure (at least on an open dense subset). Its twistor space $Z$ is the total space of the line bundle $O(2m+k,-2m-k)$ on $\oP^1\times \oP^1 - \ol\Delta$ with the zero section removed. At any point $x\in \sM_{k,m}$ the characteristic curve of the pluricomplex structure is the spectral curve of the corresponding monopole.
\label{examples}\end{example}

\begin{remark} If a pluricomplex (or $O(-1)$-) structure on $M$ is strongly integrable, then we obtain a canonical complex thickening $M^\cx$ of $M$ by taking all curves in $\sC_1$, not just the $\tau$-invariant. Conversely, we may take an arbitrary complex thickening of $M^\cx$ and extend the pluricomplex structure analytically to (an open neighbourhood of $M$ in) $M^\cx$. We obtain characteristic curves $S_m$ for any $m\in M^\cx$.
\par
Both constructions give us a coincidence space $Y^\prime\subset M^\cx\times Z$ defined as \eqref{Y}, but with $M$ replaced by $M^\cx$. This complexified point of view is the one we shall adopt in the next section. We shall usually write $Y$, instead of $Y^\prime$.\label{complexify}
\end{remark}

\section{Vector bundles with connections on strongly $O(-1)$-manifolds\label{bundles}}

Let $M$ be a smooth manifold equipped with a strongly  integrable $O(-1)$-structure. As explained in Remark \ref{complexify} we have the double fibration of complex manifolds
\begin{equation} M^\cx \stackrel{\nu}{\longleftarrow} Y \stackrel{\eta}{\longrightarrow} Z,\label{DF}\end{equation}
where $\nu$ and $\eta$ are restrictions of the natural projections from $M^\cx\times Z$. For every $m\in M^\cx$, $\nu^{-1}(m)=\{m\}\times S_m$, where $S_m$ is the characteristic curve of the $O(-1)$-structure on $T_m M^\cx$.
\par
As we shall see, a large part of the twistor theory of vector bundles on hypercomplex manifolds goes through for strongly integrable $O(-1)$-manifolds.
\par
Recall \cite{H-M} that a holomorphic vector bundle $E$ on $Z$ is called {\em weakly $M$-uniform} if $h^0(S_m, E|_{S_m})$ is globally constant on $M^\cx$. To such a bundle we can associate a locally free sheaf $\wh{\sE}=\nu_\ast\mu^\ast\sE$ on $M^\cx$. The fibre of the associated vector bundle $\wh E$ is $H^0(S_m, E|_{S_m})$ at each $m$. If the bundle $E$ is equipped with an antiholomorphic automorphism $\sigma$ covering $\tau:Z\to Z$ with square $\pm 1$, then $\wh E$ ha a real or quaternionic structure. Furthermore,  $\wh E$ comes with additional geometric structure, described by the generalised Penrose-Ward transform \cite{W,M,E,B-E,H-M}. We shall be interested in bundles $E$, such that $\wh E$ has a canonical linear connection.
\par 
We consider the bundle $\Omega^\ast_\eta Y$ of $\eta$-vertical holomorphic forms on $Y$, i.e. the exterior algebra of $\Omega^1Y/\eta^\ast\Omega^1Z$. It fits into the short exact sequence
\begin{equation} 0\to \sN_Y^\ast \to \nu^\ast \Omega^1 M^\cx \to \Omega^1_\eta Y\to 0,\label{seq1}\end{equation}
where $\sN_Y$ is (the sheaf of sections of) the normal bundle of $Y\subset M\times Z$. 
Restricted to $\nu^{-1}(m)\simeq S_m$, this sequence is the dual of the restriction of  \eqref{F0} to $S=S_m$, i.e. dual to
\begin{equation} 0\to \sW_S\to T_m^\cx M\otimes  \sO_S\to \sF\to 0,\label{G}\end{equation}
where the stalk of $\sW_S$ at any $(\zeta,\eta)\in S$ is $T^{1,0}_\zeta M+ \wh{T}^{1,0}_\eta M\subset T_m^\cx M$. 
\par
We consider the composition
\begin{equation*} d_\eta: \sO_Y\to \Omega^1Y\to \Omega^1_\eta Y.\end{equation*}
If $\sE$ is a locally free sheaf on $Z$, then $d_\eta$ extends to a relative connection on $\eta^\ast E$, i.e. a first order differential operator
\begin{equation} \nabla_\eta: \eta^\ast \sE\to \eta^\ast \sE\otimes \Omega^1_\eta Y,\label{nabla_eta}\end{equation}
which annihilates the subsheaf $\eta^{-1}\sE\subset \eta^\ast \sE$.\par
If $E$ is weakly $M$-uniform, then the Penrose-Ward transform allows us to define differential operators on $M^\cx$ by pushing down $ \nabla_\eta$. Tensoring \eqref{seq1} with $\eta^\ast \sE$
and taking the long exact sequence of direct image sheaves yields
\begin{multline} 0\to \nu_\ast\bigl( \eta^\ast \sE\otimes \sN_Y^\ast\bigr) \to  \wh\sE\otimes \Omega^1 M^\cx\to \nu_\ast\bigl(\eta^\ast \sE\otimes\Omega^1_\eta Y\bigr)\to \\
\to\nu_\ast^1\bigl(\eta^\ast \sE\otimes \sN_Y^\ast\bigr)\to \bigl(\nu_\ast^1\eta^\ast\sE\bigr) \otimes \Omega^1 M^\cx \to \nu_\ast^1\bigl(\eta^\ast \sE\otimes\Omega^1_\eta Y\bigr) \to 0.
 \label{seq2}\end{multline}
Both $\nu_\ast\bigl( \eta^\ast \sE\otimes \sN_Y^\ast\bigr)$ and $\nu_\ast^1\bigl( \eta^\ast \sE\otimes \sN_Y^\ast\bigr)$ are locally free on $M^\cx$ with the fibres at $m$ isomorphic to $H^0(S_m, E|_{S_m}\otimes \sF^\ast)$ and $H^1(S_m, E|_{S_m}\otimes \sF^\ast)$, respectively. The following fact is almost immediate and well-known.
\begin{proposition} Let $E$ be a weakly $M$-uniform holomorphic vector bundle on $Z$ such that, for every $m\in M^\cx$, $h^0(S_m, E|_{S_m}\otimes \sF^\ast)=0$ and the  map 
$$H^1(S_m, E|_{S_m}\otimes \sF^\ast)\longrightarrow   H^1(S_m,E|_{S_m})\otimes (T_m M^\cx)^\ast$$
induced by the dual to \eqref{G} is injective. Then the pushdown of $\nabla_\eta$ defines a canonical connection on the vector bundle $\wh E$. Furthermore, if  $E$ is equipped with an antilinear automorphism $\sigma$ covering $\tau:Z\to Z$ with square $\pm 1$, then the induced real or quaternionic structure on $\wh E$ is parallel.\label{connex}\end{proposition}
\begin{proof}  The assumptions imply that $\wh\sE\otimes \Omega^1 M^\cx\to \nu_\ast\bigl(\eta^\ast \sE\otimes\Omega^1_\eta Y\bigr)$ in \eqref{seq2} is an isomorphism.
\end{proof}
We shall describe two classes of vector bundles for which the assumptions of this proposition are satisfied. First of all, 
we consider the case $\sE=\sO_Z$.
We have
\begin{lemma} The map $\Omega^1 M^\cx\to \nu_\ast\Omega^1_\eta Y$  is an isomorphism.\label{isom}\end{lemma}
\begin{proof} It is enough to show that the corresponding map of vector bundles is an isomorphism at any $m\in M^\cx$.  Setting $\sE=\sO_Z$ and restricting  \eqref{seq2} to $S=S_m$ yields 
\begin{equation} 
0\to (T_m M^\cx)^\ast \to H^0(\sW_S^\ast)\to H^1(\sF^\ast)\to (T_mM^\cx )^\ast\otimes H^1(\sO_S)\to H^1(\sW_S^\ast)\to 0.\label{longG}\end{equation}
 Thus the map is injective, and it is enough to show that the dimensions of the two vector spaces are the same. We consider the free resolution \eqref{F} of $\sF$ tensored with $\sO(p,p)$:
\begin{equation}  0\to \sO(p-1,p)\otimes\cx^n\oplus \sO(p,p-1)\otimes\cx^n \to \sO(p,p)\otimes \cx^{2n} {\to}   \sF(p,p)\to 0,\label{pp}\end{equation}
where $2n=\dim M$.
It implies that, for $p\geq 0$, the induced map $H^0(\oP^1\times \oP^1-\ol\Delta, \sO(p,p))\to H^0(\sF(p,p))$ is surjective. Therefore, so is the map $H^0(S,\sO_S(p,p))\to H^0(\sF(p,p))$. Tensoring \eqref{G} with $\sO_S(p,p)$ and taking the long exact sequence, we obtain that 
$$ 0\to H^1(\sW_S(p,p))\to H^1(\sO_S(p,p))\otimes \cx^{2n}\to H^1(\sF(p,p))\to 0$$
is exact. If $p\geq 0$, then $H^1(\sF(p,p))=0$ by \eqref{pp}, and so $h^1(\sW_S(p,p))=2n h^1(\sO_S(p,p))$. We now compute the dimension of $H^0(\sW_S^\ast)$ as follows (given that the canonical bundle $K_S$ of a curve of bidegree $(k,k)$ is $\sO_S(k-2,k-2)$):
$$ h^0(\sW_S^\ast)=h^1(\sW_S\otimes K_S)= h^1(\sW_S(k-2,k-2))=2n h^1(\sO_S(k-2,k-2))=2nh^1(K_S)=2n,$$
and so the dimensions of $H^0(\sW_S^\ast)$  and $(T_m M^\cx)^\ast $ are the same.
\end{proof}
\begin{corollary} Let $E$ be a holomorphic vector bundle on $Z$, trivial on each $S_m$. Then the induced vector bundle $\wh E$ on $M^\cx$ has a canonical linear connection obtained by pushing down  $\nabla_\eta$.  \label{C1}\end{corollary}
\begin{proof} Follows immediately from Lemma \ref{isom} and Proposition \ref{connex}.
\end{proof}
The second class of vector bundles, for which the assumptions of Proposition \ref{connex} are satisfied, is obtained by an appropriate  choice of their degree. Recall that the genus $g$ of the characteristic curve of an $O(-1)$-structure of degree $k$ is $(k-1)^2$.
\begin{proposition} Let $k$ be the degree of the $O(-1)$-structure on $M$ and let $E$ be a weakly $M$-uniform holomorphic vector bundle on $Z$ of rank $l$, such that the degree of $E|_{S_m}$ is  $2l(k^2-k)$ and $h^0(S_m,\sE|_{S_m}\otimes \sF^\ast)=0$ for every $m\in M^\cx$.   Then the induced vector bundle $\wh E$ (of rank $lk^2$) on $M^\cx$ has a canonical linear connection obtained by pushing down  $\nabla_\eta$.\label{C2}\end{proposition} 
\begin{proof} The degree of $\sE|_{S_m}\otimes \sF^\ast$ is $lr(g-1)$, where $r=n/k$ is the rank of $\sF$. Therefore $h^1(S_m,\sE|_{S_m}\otimes \sF^\ast)=0$ as well.\end{proof}

\begin{corollary} Suppose that $\sF^\ast\otimes \sF$ is  trivial over each $S_m$. If $L$ is a holomorphic line bundle  over $Z$ such that $L|_{S_m}\in J^{g-1}(S_m)-\Theta(S_m)$ for each $m\in M^\cx$, then the bundle induced from $L\otimes N_Y$ is equipped with a canonical linear connection.\label{C3}
\end{corollary}
\begin{proof} Since $L|_{S_m}\not\in \Theta(S_m)$, $h^0(S_m,\sL|_{S_m}\otimes \sF\otimes \sF^\ast)=0$. The degree of $L|_{S_m}\otimes \sF$ is $rk^2+(k^2-2k)r=2r(k^2-k)$, so the previous proposition applies.
\end{proof}
Observe that the fibre of $\wh{L\otimes N_Y}$ at each $m\in M^\cx$ is $H^0(S_m, L|_{S_m}\otimes \sF)$ and is $kn=rk^2$-dimensional.
\par 
We remark that in Corollaries \ref{C1} and \ref{C3}, and in Proposition \ref{C2},  if  $E$ is equipped with an antilinear automorphism $\sigma$ covering $\tau:Z\to Z$ with square $\pm 1$, then the induced real or quaternionic structure on $\wh E$ is parallel.

\section{Pluricomplex geometry of hyperbolic monopoles: existence\label{hyperbolic}}

Let $H^3$ be the hyperbolic space with scalar curvature $-1$. A solution to Bogomolny equations  on $H^3$
\begin{equation} F_A=\ast d_A\Phi,\label{Bogo}
\end{equation}
where $A$ is an $SU(2)$-connection and  $\Phi$ an $\su(2)$-valued Higgs field, is called a hyperbolic monopole if it has finite energy and the norm of the Higgs field has limit $m>0$ at infinity. This constant $m$ is called {\em mass} of the monopole.  The {\em charge} of a monopole is the topological degree of $\Phi$. The gauge group acts on the space of monopoles of fixed charge $k$ and mass $m$ and the {\em moduli space $\sM_{k,m}$ of framed monopoles} of charge $k$ and mass $m$ is obtained by quotienting by gauge transformations which are $1$ at infinity. 
\begin{remark} It is easy to see, via rescaling, that monopoles of mass $m$ correspond to monopoles satisfying $\|\Phi\|\to 1$ on the hyperbolic space with scalar curvature $-1/m^2$.
\end{remark}

\medskip

We shall now assume that $m\in \frac{1}{2}\oN$, and show that $\sM_{k,m}$ carries a natural pluricomplex structure (at least on an open dense subset), closely related to the well-known hyperk\"ahler structure on the moduli space of Euclidean monopoles. 
\par
 We begin by recalling the twistorial construction of $\sM_{k,m}$ (cf. \cite{A,MS1,MS2,Nash}). A hyperbolic monopole  corresponds to an algebraic curve $S$ in $\oP^1\times \oP^1-\ol{\Delta}$. This curve, called the {\em spectral curve} of the monopole, has bidegree $(k,k)$, no multiple components, is invariant under the involution \eqref{sigma}, and satisfies additional conditions, the most important of which says that the line bundle $L^{2m+k}|_S=O(2m+k,-2m-k)|_S$ is trivial ($L=O(1,-1)$). The manifold $\sM_{k,m}$ can be identified with $\{${\em spectral curves $+$ real trivialisation of $L^{2m+k}|_S$$\}$}. Here, a trivialisation $s$ is real if $s\sigma^\ast s=1$. The idea of Nash \cite{Nash} is then to lift $S$ to the total space $Z$ of $L^{2m+k}-\{0\}$ using the given trivialisation. Thus, points of $\sM_{k,m}$ can be considered as algebraic curves  in $Z$. Nash identifies the normal bundle of such a curve $\hat S$ in $Z$ as
\begin{equation} \wh N\simeq \tilde{E}L^m(k,0)|_S,\label{tildeE}\end{equation}
where $\tilde{E}$ is the holomorphic rank $2$ bundle on $\oP^1\times \oP^1-\ol{\Delta}$ corresponding to the given monopole.
Writing $N$ for the pullback of $\wh N$ to $S$,  Nash also shows that $N(-1,-1)$ has trivial cohomology and, consequently $h^1(S,N)=0$, $h^0(S,N)=4k$. Thus, we are in the situation of Theorem \ref{twistor} and can already conclude that the moduli space of hyperbolic monopoles has a $\oP^1$ of complex structures, which is a strongly integrable (hence integrable) $O(-1)$-structure.
\par
At a given point $x$ of  $\sM_{k,m}$, this  $O(-1)$-structure is pluricomplex provided that
\begin{equation} h^0(S,N(-2,-0))=h^0(S,N(0,-2))=0,\label{vanish?}\end{equation}
where $S$ is the spectral curve of $x$.
\begin{remark} Nash \cite{Nash} claims  to have shown that there are {\em natural} identifications $H^0(S,N)\simeq H^0(S,N(-1,0))\otimes H^0(S,\sO(1,0))$ and $H^0(S,N)\simeq H^0(S,N(0,-1))\otimes H^0(S,\sO(0,1))$. For this, however, one needs \eqref{vanish?}, which is not proved anywhere in \cite{Nash}. In fact, Nash states explicitly \cite[Remark 3.4]{Nash} that his proof of vanishing of $h^0(S,N(-1,-1))$ does not generalise to \eqref{vanish?}.  \label{Nash-gap}\end{remark}
\begin{theorem} For any $k>0$ and $m\in \frac{1}{2}\oN$, the vanishing condition  \eqref{vanish?} holds on an open dense subset of $\sM_{k,m}$.\label{final}
\end{theorem}
\begin{remark} As we show in the Appendix, as soon as \eqref{vanish?} holds, it follows automatically that $N(p,q)$ has no cohomology for any $p,q\in \oZ$ with $p+q=-2$.
\end{remark}
\begin{proof} Let us describe the point at which Nash's proof of vanishing of $H^\ast(N(-1,-1))$ breaks down for $H^\ast(N(-2,0))$. Let $Q=\oP^1\times \oP^1$.
We recall that we have the following exact sequence
\begin{equation} 0\to \sO_S\to N\to \sO_S(k,k)\to 0,\label{e1}
\end{equation}
from which we obtain embeddings $H^0(N(-1,-1))\hookrightarrow H^0(\sO_S(k-1,k-1))$ and $H^0(N(-2,0))\hookrightarrow H^0(\sO_S(k-2,k))$. In addition, we have the exact sequence
\begin{equation} 0\to \sO_Q\to \sO_Q(k,k)\to \sO_S(k,k)\to 0,\label{e2} 
\end{equation}
from which it follows that $H^0(\sO_S(k-1,k-1))\simeq H^0(\sO_Q(k-1,k-1))$ and so every section of $H^0(N(-1,-1))$ can be represented by a polynomial in two variables. 
\par
If we tensor \eqref{e2} with $\sO(-2,0)$ instead, the long exact sequence yields:
\begin{equation*} 0\to H^0(\sO_Q(k-2,k))\to H^0(\sO_S(k-2,k))\to H^1(\sO_Q(-2,0))\to 0.\label{e3} 
\end{equation*}
This time, the $1$-dimensional space $H^1(\sO_Q(-2,0))$ is an obstruction to extending sections of $\sO_S(k-2,k)$ to $Q$. However, if $s\in H^0(N(-2,0))$ is such that its image under the embedding $H^0(N(-2,0))\hookrightarrow H^0(\sO_S(k-2,k))$ can be extended to $Q$, then the remainder of Nash's proof goes through for $s$ (with easy modifications), and we can conclude that  $s=0$.  Consequently, $h^0(N(-2,0))\leq 1$, and moreover a nonzero  section of $N(-2,0)$ maps to a section of $\sO_S(k-2,k)$, which does not extend to $Q$.
\par
We now recall, from Lemma \ref{-2,0}, that $H^0(S,N(-2,0))\neq 0$ is equivalent to the existence of a $v\in T_x\sM_{k,m}$ ($x$ denotes the hyperbolic monopole, the spectral curve of which is $S$) such that, say, $J_0v=J_1v$, i.e. to the vanishing of $\det(J_0J_1^{-1}-1)$ at $p$. Since this is a real-analytic function on a connected real-analytic manifold $\sM_{k,m}$, we shall have proved the theorem as soon as we show that $H^0(S,N(-2,0))=0$ at one point of $\sM_{k,m}$. For this, we consider an axially symmetric monopole, discussed in detail in section 7 of \cite{MNS}. Its spectral curve $S$ is the union of $k$ rational curves of bidegree $(1,1)$, given by the equation
\begin{equation} P(\zeta,\eta)= \prod_{i=1}^k (\eta-a_i\zeta)=0.\label{axi}\end{equation}
The condition $L^{2m+k}|_S\simeq \sO$ translates into $a_i^{2m+k}=a_j^{2m+k}$ for all $i,j$. 
For such a curve, we are  going to identify the connecting homomorphism $\lambda:H^0(\sO_S(k-2,k))\to H^1(\sO_S(-2,0))$ corresponding to \eqref{e1} tensored by $\sO(-2,0)$ and show that any section of $\sO_S(k-2,k)$, which does not come from $\sO_Q(k-2,k)$, cannot belong to $\Ker\lambda$. Owing to the above discussion, this will prove the theorem.
\par
We need to compute the extension class of $N$ in \eqref{e1} for the curve \eqref{axi}. According to Nash \cite[Lemma 2.2]{Nash}, it is given  by $\delta(\phi)|_S\in H^1(S,L^{2m+k}(-k,-k))$, where $\phi$ trivialises $L^{2m+k}|S$ and $\delta$ is the connecting homomorphism associated to the following sequence of sheaves on $Q$
\begin{equation*} 0\to L^{2m+k}(-k,-k)\to L^{2m+k}\to L^{2m+k}|_S\to 0.\end{equation*}
The bundle $L$ can be trivialised on the open cover $U_0=\{\zeta,\eta\neq\infty\}$, $U_1=\{\zeta,\eta\neq 0\}$ of $Q\backslash \ol\Delta$ with transition function $\eta/\zeta$. Let $\beta=a_i^{2m+k}$, $i=1,\dots,k$. A section $\phi$ of $L^{2m+k}|_S$ can be then written as
$$  \phi(\zeta,\eta)\equiv 1\enskip\text{if $(\zeta,\eta)\in U_0\cap S$}, \quad \phi(\zeta,\eta)\equiv \beta\enskip\text{if $(\zeta,\eta)\in U_1\cap S$}.$$
The \v{C}ech representative of $\delta(\phi)$ on $U_0\cap U_1$  is then (see the proof of Lemma \ref{homol} for justification): 
$$ \delta(\phi)=\frac{1-(\zeta/\eta)^{2m+k}\beta}{P(\zeta,\eta)}.
$$
We can rewrite this as
$$\left(\frac{\zeta}{\eta}\right)^{2k+m}\frac{(\eta/\zeta)^{2k+m}-a_i^{2m+k}}{\zeta^k\prod_{j=1}^k(\eta/\zeta-a_j)}= \frac{1}{\zeta^k} \left(\frac{\zeta}{\eta}\right)^{2k+m}
\frac{\sum_{r=0}^{2m+k-1}(\eta/\zeta)^r a_i^{2m+k-1-r}}{\prod_{j\neq i}^k(\eta/\zeta-a_j)},
$$
and, hence, the restriction $\delta_i$ of $\delta(\phi)|_S$ to the component $S_i=\{(\zeta,\eta);\eta=a_i\zeta\}$ is
$$ \delta_i=\delta(\phi)|_{S_i}=
\frac{1}{\zeta^k}\frac{2m+k}{a_i\prod_{j\neq i}(a_i-a_j)}.$$
 The connecting homomorphism $\lambda:H^0(\sO_S(k-2,k))\to H^1(\sO_S(-2,0))$ is now given on each $S_i$ by the multiplication by $\delta_i$. We identify sections of $\sO_S(k-2,k)$ with sections of $\sO_S(k-1,k)$ which vanish on the divisor $D_0=\{\zeta=0\}$.
All sections of $\sO_S(k-1,k)$ extend to $Q$ (by tensoring \eqref{e2} with $\sO(-1,0)$ and taking the long exact cohomology sequence), and therefore they can be written as polynomials of bidegree $(k-1,k)$. The sections vanishing on $D_0$ are polynomials of the form
$$ b\eta^k+\zeta\left(\sum_{r=0}^{k-2}\sum_{s=0}^k p_{rs}\zeta^r\eta^s\right).$$
We can, therefore, write any section of $\sO_S(k-2,k)$ as
\begin{equation} b\frac{\eta^k}{\zeta}+\sum_{r=0}^{k-2}\sum_{s=0}^k p_{rs}\zeta^r\eta^s.\label{section}\end{equation}
Observe that sections extending to $Q$ are those with $b=0$.  Multiplying the section \eqref{section} by $\delta_i$ and restricting to $S_i$ produces a Laurent polynomial, the $1/\zeta$-term of which is
$$ \frac{1}{\zeta}  \frac{2m+k}{a_i\prod_{j\neq i}(a_i-a_j)}\left(ba_i^k+\sum_{r+s=k-1}p_{rs}a_i^s\right).
$$
This term must vanish if the corresponding class in $H^1(S_i,\sO(-2,0))$ is to be zero. Therefore
$$ ba_i^k+\sum_{r+s=k-1}p_{rs}a_i^s=0, \enskip \forall i=1,\dots,k,
$$
and, hence,
$$ bz^k+\sum_{r+s=k-1}p_{rs}z^s=b\prod_{i=1}^k(z-a_i).
$$ 
However, $r\leq k-2$, and so the constant term on the left-hand side is zero. As the $a_i$ are non-zero, this implies that $b=0$, and we finally conclude that any section in $\Ker\lambda$  does extend to $Q$.
\end{proof}

\begin{remark} We expect of course that \eqref{vanish?} holds for every hyperbolic monopole, and so, that the $O(-1)$-structure on $\sM_{k,m}$ is pluricomplex everywhere. As observed in the course of the proof, Nash's results imply already that $h^0(S,N(-2,0))\leq 1$ and that a nonzero section of $N(-2,0)$ cannot belong to the kernel of the composition of natural maps
$$ H^0(S,N(-2,0))\to H^0(S,\sO(k-2,k))\to H^1(Q,\sO(-2,0)).$$
Perhaps a more sophisticated variation of our argument would show this to be impossible. 
\end{remark}

\begin{remark}  The construction of a natural metric, given in \S\ref{plurimflds}, works for $\sM_{k,m}$. The $O(2)$-valued 2-form $\Omega$ has been defined by Nash \cite[Lemma 3.5]{Nash}. Whether the resulting $(2,0)$-tensor $g$ is positive-definite and, if so, what is its   physical significance, remains to be investigated. 
\end{remark}

\section{Pluricomplex geometry of hyperbolic monopoles: description\label{hyperbolic2}}

We shall now describe the pluricomplex or $O(-1)$-structure on $\sM_{k,m}$ ($m\in \frac{1}{2}\oN$) from several different points of view: via rational maps and scattering of monopoles, as acting on solutions to linearised field equations, and by viewing hyperbolic monopoles as invariant instantons. Finally, we shall discuss in this context Hitchin's self-dual Einstein metrics  \cite{Hit}.

\subsection{Scattering and rational maps\label{scatter}} Let $K:\oP^1\to \sJ(\sM_{k,m})$ be the natural pluricomplex structure on the moduli space of hyperbolic monopoles defined in the previous section. 
The proof of Theorem \ref{twistor} shows that the complex structure $-K(\zeta_0)$ is obtained via identification of $\sM_{k,m}$ with the space $\Rat_k(\oP^1)$ of based rational maps defined as follows.  If the spectral curve $S$ of an  $x\in \sM_{k,m}$ is given by an equation
\begin{equation*} \sum_{i,j=0}^k a_{ij}\zeta^i\eta^j=0\label{Seq}\end{equation*}
in $\oP^1\times \oP^1$, then the poles of the rational map $p(\eta)/q(\eta)=\Phi_{\zeta_0}(x)$  are given by the roots $\eta_1,\dots,\eta_k$ of the polynomial
\begin{equation*} \sum_{i,j=0}^k a_{ij}\zeta_0^i\eta^j=0,
\end{equation*}
while the values of the numerator $p(\eta)$ at $\eta_1,\dots,\eta_k$ are given by the values of the section of $L^{2k+m}$ corresponding to the monopole $x$ (in a fixed trivialisation).
\par
On the other hand, we know from \cite{A}, that this identification $\Phi_\zeta:\sM_{k,m}\to \Rat_k(\oP^1)$ is obtained from solutions of the scattering equation
\begin{equation} (\D^A_t-i\Phi)s=0\label{scatt}\end{equation}
along geodesics beginning at a point $\zeta$ of the ideal boundary $S^2\simeq \oP^1$. Consider \eqref{scatt} along the geodesic beginning at $\zeta$ and ending at $-1/\bar{\eta}$. The space of solutions to  \eqref{scatt} is $2$-dimensional and a basis $s_1(t),s_2(t)$ can be chosen by requiring that $\lim_{t\to+\infty}e^{mt}s_1(t)=1$, $\lim_{t\to+\infty}e^{-mt}s_2(t)=1$, $s_2(t)\perp s_1(t)$. A solution $r(t)$ to \eqref{scatt}, which decays as $t\to -\infty$, is a linear combination of $s_1(t)$ and $s_2(t)$: $r(t)=as_1(t)+bs_2(t)$. Then the rational function $\Phi_\zeta(x)$ corresponding to $x$ has value $a/b$ at $\eta\in \oP^1-\{-1/\bar\zeta\}$.

\begin{remark} Both descriptions of $K:\oP^1\to \sJ(\sM_{k,m})$ show that the action on $\sM_{k,m}$ of the isometry group $PSL(2,\cx)$ of $H^3$ is compatible with the natural action of $PSL(2,\cx)$ on $\oP^1$, i.e. $g.K(\zeta)|_x=K(g(\zeta))|_{g(x)}$.\label{intertwine}
\end{remark}

\subsection{Pluricomplex structure in terms of fields} We can identify the pluricomplex structure of $\sM_{k,m}$ directly in terms of solutions to Bogomolny equations. Our description is rather informal, and we leave aside the interesting question which Hilbert or Banach spaces are needed. In principle (modulo delicate analytical issues), this description works for arbitrary mass $m$. It also remains valid for $SU(N)$-monopoles ($N>2$), so that the moduli space of $SU(N)$-monopoles should also be  pluricomplex manifold.
\par
Let $(A,\Phi)$ be a hyperbolic monopole. We describe the tangent space to $\sM_{k,m}$ at $(A,\Phi)$.  Linearising Bogomolny equations gives the following equation
\begin{equation} d_A a=\ast(d_A\phi +[a,\Phi]), \label{lin1}
\end{equation} 
for an $\su(2)$-valued $1$-form $a$ and $\su(2)$-valued function $\phi$ on $H^3$. The tangent space $T_{(A,\Phi)}\sM_{k,m}$ is identified with solutions to \eqref{lin1} modulo infinitesimal gauge transformations (i.e. solutions to \eqref{lin1} generated by gauge transformations equal to $1$ on the ideal boundary). The previous subsection shows that the pluricomplex structure is induced by that on $H^3\times S^1$, which in turn arises as in Example \ref{basic}. If we fix a point $\zeta$ on the ideal boundary of $H^3$ and identify $H^3$ with the upper-half space, then the metric on $H^3\times S^1$ becomes
$$ e^{2t}(dx^2+dy^2)+dt^2+\nu^2,$$
where $\nu={\rm d}\ln\lambda$ is an invariant $1$-form on $S^1$. As in  Example \ref{basic}, the complex structure $J_\zeta$ is given by 
\begin{equation} J_\zeta(dx)=-dy,\quad J_\zeta(dt)=-\nu.\label{jzero}\end{equation}
In coordinates $x,y,t$, the Bogomolny equations become:
$$ \D^A_x\Phi=\left[\D^A_y,\D^A_t\right],\quad \D^A_y\Phi=\left[\D^A_t,\D^A_x\right],\quad \D^A_t\Phi=e^{-2t}\left[\D^A_x,\D^A_y\right].
$$
Let us write $a=a_1dx + a_2 dy+a_3 dt$. The linearisation of Bogomolny equations is:
\begin{equation}\D_x\phi+[A_1,\phi]+[a_1,\Phi]=\D_ya_3-\D_ta_2+[A_2,a_3]+[a_2,A_3],\label{B1}
\end{equation}
\begin{equation}\D_y\phi+[A_2,\phi]+[a_2,\Phi]=\D_ta_1-\D_xa_3+[A_3,a_1]+[a_3,A_1],\label{B2}
\end{equation}
\begin{equation} \D_t\phi+[A_3,\phi]+[a_3,\Phi]=e^{-2t}\left(\D_xa_2-\D_ya_1+[A_1,a_2]+[a_1,A_2]\right).\label{B3}
\end{equation}
We now consider the almost complex structure \eqref{jzero} acting on the space of all $S^1$-invariant $\su(2)$-valued $1$-form $a+\phi\nu$ on $H^3\times S^1$. It is easy to observe that \eqref{B1} and \eqref{B2} are exchanged, while \eqref{B3} becomes:
\begin{equation}e^{-2t}\left(\D_xa_1+\D_y a_2+[A_1,a_1]+[A_2,a_2]+[\Phi,\phi]\right)+\D_ta_3+[A_3,a_3]=0.\label{B5}
\end{equation}
\begin{remark} The reader may be interested in comparing this with the condition of orthogonality (with respect to the $L^2$-metric) to compactly supported gauge transformations: 
\begin{equation} \ast d_A\ast a +[\Phi,\phi]=0, \label{lin2}\end{equation}
which in these coordinates becomes 
\begin{equation}e^{-2t}\left(\D_xa_1+\D_y a_2+[A_1,a_1]+[A_2,a_2]+[\Phi,\phi]\right)+\D_ta_3+[A_3,a_3]+2a_3=0.\label{B4}\end{equation}
\end{remark}
We do not know whether the space of solutions to equations \eqref{B1}-\eqref{B5} is isomorphic to $T_{(A,\Phi)}\sM_{k,m}$, i.e. whether there exist nontrivial solutions to \eqref{B5} generated by gauge transformations equal to $1$ on the ideal boundary. Nevertheless, we can describe $J_\zeta$ acting on $T_{(A,\Phi)}\sM_{k,m}$ as follows.
\par
Let us write $X_{(A,\Phi)}$ for the space of solutions to \eqref{B1}-\eqref{B3}, and let $\sL\subset X_{(A,\Phi)}$ be those generated by gauge transformations equal to $1$ on the ideal boundary (we stress once again that this needs to be formalised by specifying the relevant Hilbert or Banach spaces). Similarly, let $Y_{(A,\Phi)}$ for the space of solutions to \eqref{B1}-\eqref{B3} and \eqref{B5} and $\sK=\sL\cap Y_{(A,\Phi)}$. Then $T_{(A,\Phi)}\sM_{k,m}\simeq X_{(A,\Phi)}/\sL\simeq Y_{(A,\Phi)}/\sK$. Take now an equivalence class of $(a,\phi)$ in $ Y_{(A,\Phi)}/\sK$; identify it with an $S^1$-invariant $\su(2)$-valued $1$-form $a+\phi\nu$ on $H^3\times S^1$; act on it via \eqref{jzero} and map the result to $X_{(A,\Phi)}/\sL$. This is $J_\zeta$.

\subsection{Hyperbolic monopoles and instantons on $S^4$ \label{instantons}}

One can view  hyperbolic monopoles as invariant instantons on $\oR^4$, or equivalently, invariant instantons on $S^4$, framed at $\infty$. This is the point of view adopted, for example, by Atiyah \cite{A} and Braam and Austin \cite{BA}. Framing an instanton at a point of $S^4$ immediately picks a point on the ideal boundary of $H^3$ and so it is not compatible with the full pluricomplex structure. In fact, as discussed in section \ref{scatter}, Atiyah \cite{A} describes one of the complex structures: precisely the one corresponding to the chosen point on the ideal boundary of $H^3$.
\par
We are going to show that one can see the full pluricomplex structure, provided one gives a different framing to homogeneous instantons.

We begin with the Penrose fibration:
\begin{equation} \pi:\oC\oP^3\to \oH\oP^1\simeq S^4,\quad \cx(z_0,z_1,z_2,z_3)\mapsto \oH(z_0+z_1j,z_2+z_3j).\label{Penrose}\end{equation}
The fibres of $\pi$ are ``real" lines in $\cx\oP^3$, i.e. $\oP^1$-s invariant under the involution
\begin{equation}\sigma(z_0,z_1,z_2,z_3)=(-\bar z_1,\bar z_0, -\bar z_3,\bar z_2)\label{sigma3}\end{equation}
The preferred point $\infty$ corresponds to the line $l_\infty=\{[z_0,z_1,0,0]\}$. The hypercomplex structure of $\oR^4$ can be now seen as follows. Choose a $P\simeq \cx\oP^2$ containing $l_\infty$ (there are $\oP^1$ of such planes). Any real line different from $l_\infty$ meets $P$ in a unique point and the restriction of $\pi$ to $P- l_\infty$ is an isomorphism $\cx^2\simeq \oR^4$. The group $PGL(2,\oH)$ is exactly the subgroup of $PSL(4,\cx)$ commuting with the involution $\sigma$.

 We now recall the Atiyah-Ward twistor description  of $SU(2)$ instantons (anti-self-dual connections) of charge $k$ on $S^4=\oR^4\cup \{\infty\}$, framed at $\infty$.. 
Atiyah and Ward \cite{AW} have shown that a framed instanton on $S^4$ corresponds to a holomorphic vector bundle $E$ on $\oP^3$, framed along $l_\infty$, trivial on real lines and equipped with a quaternionic structure covering $\sigma$. An instanton has charge $k$ if $c_2(E)=k$. As shown by Donaldson \cite{Don}, restricting to $P\simeq \oC\oP^2$, chosen as above, identifies the moduli space of framed instantons with the moduli space  of holomorphic vector bundles of rank $2$ on $\oC\oP^2$, framed along a line, and satisfying $c_1(E)=0$, $c_2(E)=k$. This exhibits, in particular, the hypercomplex structure on the moduli space of framed instantons.
\par
Let us now describe the pluricomplex structure of $H^3\times S^1$ using the Penrose fibration. We  consider the $S^1$-action on $S^4$, which comes from right multiplication by $e^{\frac{1}{2}i\theta}$ on $\oH^2$. It is clear from \eqref{Penrose} that this action lifts to an action of $\oC^\ast$ on $\oC\oP^3$ defined by:
\begin{equation}\lambda.[z_0,z_1,z_2,z_3]\mapsto \left[\lambda^{1/2}z_0,\lambda^{-1/2}z_1,\lambda^{1/2}z_2,\lambda^{-1/2}z_3\right].\label{action3}\end{equation}
It is the same action as considered by Atiyah in \cite{A}. Its fixed-point set consists of two $\oP^1$-s: $\oP^1_-=\{z_0=z_2=0\}$ and $\oP^1_+=\{z_1=z_3=0\}$. The centraliser of $S^1$ in $PGL(2,\oH)$ is $PGL(2,\cx)$ acting via
\begin{equation} A(v,w)=(Av,\bar Aw), \enskip\text{where $v=(z_0,z_2)$, $w=(z_1,z_3)$}.\label{PGL2}\end{equation}
We now interpret the pluricomplex structure of $H^3\times S^1\simeq \oR^4-\oR^2$ similarly to the above interpretation of the hypercomplex structure on $\oR^4$, i.e. by restricting to planes. This time we choose $S^1$-invariant planes $\oC\oP^2\subset \oC\oP^3$. These are those that contain either $\oP^1_+$ or $\oP^1_-$ and so they  are parameterised by $\oP^1\cup\oP^1\subset (\cx\oP^3)^\ast$. Let $P=P(p)$ be a plane containing $\oP^1_-$, and passing through the point $p\in \oP^1_+$. In other words:
\begin{equation} P=\{[wa,z_1,wb,z_3]; \enskip w,z_1,z_3\in \oC\},\quad p=[a,0,b,0]\in \oP^1_+.
\end{equation}
 Then $P$ contains a unique real $\oP^1$, $l_p=\{wa,z\bar a,wb,z\bar b\}$, which is mapped via \eqref{Penrose} to the point $\pi(p)$ lying in the fixed-point set of $S^1$. The restriction of \eqref{Penrose} to $P-l_p$ provides another isomorphism $\bar{\phi}_p:\oC^2\simeq S^4-\pi(p)\simeq \oR^4$, which depends on the choice of point $p$ (or $\pi(p)$). Each of these maps $\oP^1_-\subset P$ to $\oR^2=\bigl(\oR^4\bigr)^{S^1}$ and restricting $\bar{\phi}_p$ to $P-(l_p\cup \oP^1_-)$ gives us an isomorphism
$$\phi_p: \oR^4-\oR^2\simeq \cx \times\cx^\ast,$$
which is the same as in Example \ref{basic} (after passing from $\oR$ to $S^1\simeq \oR/\oZ$). In other words, the isomorphisms $\phi_p$ determine the canonical pluricomplex structure on $\oR^4-\oR^2\simeq H^3\times S^1$. 

\medskip

We now return to instantons. Let $E$ be a framed instanton bundle on $\oC\oP^3$, i.e. $E$ is trivialised along $l_\infty$, it is trivial on each real line and it has a quaternionic structure covering $\sigma$. If we want to induce a pluricomplex structure on the moduli space of such bundles, we would need to restrict them to $S^1$-invariant planes in $\cx\oP^3$. This, however, is not compatible with the framing:  the plane $P(p)$, $p\in \oP^1_+$, contains $l_\infty$ only if $p=[1,0,0,0]$ and otherwise  $P(p)\cap l_\infty=[0,1,0,0]$. This point does not lie on $l_p$ so there is no natural framing of $E|_{P(p)}$ along $l_p$. Therefore the canonical pluricomplex structure of $H^3\times S^1$ does not induce a pluricomplex structure on the moduli space of instantons {\em framed at a point}.
\par
Things can be made to work, however, if we restrict ourselves to $S^1$-invariant instantons, i.e. to homogeneous instanton bundles on $\oP^3$.
\par
Let $E$ be an instanton bundle on $\cx \oP^3$ invariant under action \eqref{action3}. Then, as a homogeneous bundle,
\begin{equation} E|_{\oP^1_+}\simeq O(-k)\otimes L^m\oplus O(k)\otimes L^{-m},\label{homogen}\end{equation} 
where $L$ is the standard $1$-dimensional representation of $\oC^\ast$ (cf. \cite{A}). Let $\sM(k,m)$ be the space of isomorphism classes of:\begin{center}{\em homogeneous instanton bundle $E$ $+$ an isomorphism \eqref{homogen}}.\end{center} 
Consider again any $S^1$-invariant plane $P(p)$ in $\cx\oP^3$. The isomorphism \eqref{homogen} gives a trivialisation of $E_p$, and since $E$ is trivial along $l(p)$, this trivialisation extends to a framing $E|_{P(p)}$ along $l_p$. Thus we obtain an identification of $\sM(k,m)$ with the space of {\em homogeneous} holomorphic vector bundles of rank $2$ on $\oC\oP^2$, framed along a line, and satisfying $c_1(E)=0$, $c_2(E)=k$. This gives a $2$-sphere of complex structures on $\sM(k,m)$. These complex structures can be identified as in \cite{A}: for each $p\in \oP^1_+$, we obtain a biholomorphism $\sM({k,m})\simeq \Rat_k(\oP^1)$. These agree with biholomorphisms described in section \ref{scatter}, and so $\sM({k,m})\simeq\sM_{k,m}$ as pluricomplex manifolds.


\subsection{Charge $2$ monopoles\label{quat}} We may ask how the self-dual Einstein metrics of Hitchin \cite{Hit}, defined on the moduli space of {\em centred} framed monopoles of charge $2$ fit into this scheme. Certainly, manifolds with pluricomplex structure of degree $2$ admit, under mild conditions, a canonical quaternionic structure. To begin, we recall \cite{MNS} that the line bundle $\sO(1,-1)$ on $\oP^1\times \oP^1-\ol\Delta$ can be raised to an arbitrary complex power and, consequently, we have a line bundle $\sO(\lambda,\mu)$ on $\oP^1\times \oP^1-\ol\Delta$ for every $\lambda,\mu\in \cx$ such that $\lambda+\mu\in \oZ$. Moreover, $\sO(\lambda,\mu)$ extends to $\oP^1\times \oP^1$ if and only if both $\lambda$ and $\mu$ are integers. If, now, the characteristic curve $S$ (which is elliptic if $k=2$)  of the pluricomplex structure on $V=T_xM$ satisfies 
\begin{equation} \sO_S(1/2,-1/2)\neq \sO_S,\label{assumpt}\end{equation}
then $h^0\bigl(S, \sO_S(1/2,1/2)\bigr)=2$, and, consequently, we have the following exact sequence on $S$:
$$ 0\to \sF(-1,-1)\to \sF(-1/2,-1/2)\otimes H^0(S,\sO(1/2,1/2))\to \sF\to 0,$$
and taking the long exact sequence shows (as $\sF(-1,-1)$ has no cohomology) that there is a canonical isomorphism
\begin{equation} H^0\bigl(S,\sF\bigr)\simeq H^0\bigl(S,\sF(-1/2,-1/2)\bigr)\otimes H^0(S,\sO_S(1/2,1/2))\label{qq}\end{equation}
compatible with real structures. We observe that $\sO_S(1/2,1/2)$ is not a $\sigma$-sheaf. Instead, there is an isomorphism $\phi:\sO_S(1/2,1/2) \stackrel{\sim}{\rightarrow} \sigma^\ast \sO_S(1/2,1/2)$ such that $\sigma^\ast(\phi)\phi=-1$. It induces a quaternionic structure on $ H^0(S,\sO(1/2,1/2))$ i.e., an antiholomorphic endomorphism with square $-1$.
\par
Therefore, both factors on the right-hand side of \eqref{qq} have a quaternionic structure, the tensor product of which is the real structure of $H^0\bigl(S,\sF\bigr)$. Thus, we obtain a canonical decomposition $V^\cx\simeq \cx^{2r}\otimes \cx^2$ and, when we vary $x$, an almost quaternionic structure on $M$ (this structure cannot be almost hypercomplex: by \cite[Lemma 9.1]{MNS}, there are no global sections of $\sO(1/2,1/2)$  on $\oP^1\times \oP^1-\ol\Delta$).
\par
The assumption \eqref{assumpt} is automatically satisfied at every point of the moduli space $\sM_{2,m}$ of hyperbolic monopoles of charge $2$ and half-integer mass $m$. 
Indeed, a non-zero section of $\sO_S(1/2,-1/2)$ would give a nonzero section of $\sO_S(1,-1)$, which is impossible by \cite[Theorem 7]{MNS}. Thus,  $\sM_{2,m}$ has a natural almost quaternionic structure.
\par
Whether this almost quaternionic structure is related to Hitchin's metrics remains to be investigated.

\section{Monopoles with mass zero \label{zero-mass}} The twistor description of hyperbolic monopoles generalises to the case $m=0$, i.e. to massless hyperbolic monopoles. These objects are related to solutions of the Yang-Baxter equation \cite{A2,AM2} and can be viewed as limits of hyperbolic monopoles when $m\to 0$ \cite{JN}.
\par
From the twistor point of view they correspond to $\sigma$-invariant curves $S$ of bidegree $(k,k)$ in $\oP^1\times \oP^1$ such that $\sO_S(k,-k)\simeq \sO_S$. As for hyperbolic monopoles, we can consider a twistor space $Z$, which is the total space of the bundle $\sO(k,-k)$ on $\oP^1\times \oP^1-\ol\Delta$ (minus the zero section). The curve $S$ of a  monopole with mass zero can be lifted a curve $\hat{S}$ to $Z$, and we are going to show that the normal bundle $N=N_{\hat{S}/Z}$ satisfies $H^\ast(\hat{S},N(-1,-1))=0$. Therefore the manifold $\sM_{k,0}$ of massless monopoles admits a natural strongly integrable $O(-1)$-structure. On the other hand, we are also going to show that, if $k\geq 2$, then $H^\ast(\hat{S},N(-2,0))\neq 0$, so this $O(-1)$-structure is not pluricomplex. In fact it is as far from pluricomplex as possible: at each point the vector bundle $ T_x^\cx M/K(x)^\ast \sV^{1,0}$ splits as $\sO(k)\oplus \sO(k)\oplus \sO^{\oplus 2k-2}$.
\par
Thus, we have the following picture of the limiting behaviour of pluricomplex geometry on $\sM_{k,m}$: as $m\to \infty$, the pluricomplex structure becomes hypercomplex, i.e. the characteristic curves (which are the spectral curves) approach the diagonal in $\oP^1\times \oP^1$ (see \cite{BS2} for more on this limiting behaviour); on the other hand, as $m\to 0$, the pluricomplex structure degenerates to an $O(-1)$-structure, which is no longer pluricomplex.
\par
Nash's proof of vanishing of $H^\ast (N(-1,-1))$ does not apply to the case $m=0$, so we need to prove:
\begin{proposition} Let $Z$ be the total space of $\Lambda=\sO(k,-k)$ on $Q=\oP^1\times \oP^1$ without the zero section. Let $\hat S\subset  Z$ be a reduced curve such that the restriction to $\hat S$ of the natural projection $Z\to Q$ is a biholomorphism onto its image, which is a curve $S$ in the linear system $|\sO(k,k)|$. Then $H^\ast(\hat S,\sN_{\hat S/Z}(-1,-1))=0$.\end{proposition}
\begin{proof} Let $\phi$ be the section of $\Lambda|_S$ so that  $\hat S=\phi(S)$, and 
let us write $N$ for the pullback of $\sN_{\hat S/Z}$ to $S$. We have an exact sequence 
\begin{equation} 0\to \Lambda|_S\to N\to \sO_S(k,k)\to 0.\label{E0}\end{equation} 
In \cite[Lemma 2.2]{Nash}, Nash identifies the extension class of $N$ as $\delta(\phi)|_S\in H^1(S,\Lambda(-k,-k))$, where $\delta$ is the connecting homomorphism associated to the following sequence of sheaves on $Q$
\begin{equation*} 0\to \Lambda(-k,-k)\to \Lambda\to \Lambda|_S\to 0.\label{E1}\end{equation*}
If we tensor \eqref{E0} with $\sO(-1,-1)$, we obtain:
 \begin{equation} 0\to \Lambda|_S(-1,-1)\to N(-1,-1)\to \sO_S(k-1,k-1)\to 0,\label{E2}\end{equation} 
and, consequently, we aim to show that the connecting homomorphism $H^0(S,\sO(k-1,k-1))\to H^1(S,\Lambda(-1,-1))$ is an isomorphism. By the above lemma of Nash, this homomorphism is given by $t\mapsto t\cdot \delta(\phi)|_S$. We need a simple homological lemma:
\begin{lemma}
Let $X$ be a complex manifold and let  $L,M,P$ be line bundles on $X$. Let $S$ be a reduced effective divisor of $M$ and consider short exact sequences
\begin{equation*} 0\to LM^\ast\to L\to L|_S\to 0,\label{eq1}\end{equation*}
\begin{equation*} 0\to PM^\ast\to P\to P|_S\to 0,\label{eq2}\end{equation*}
\begin{equation*} 0\to LPM^\ast\to LP\to (LP)|_S\to 0.\label{eq3}\end{equation*}
Denote by 
 $\delta_L:H^0(S,L)\to H^1(X,LM^\ast)$, $\delta_P:H^0(S,P)\to H^1(X,PM^\ast)$, and  $\delta_{LP}:H^0(S,LP)\to H^1(X,LPM^\ast)$
the associated connecting homomorphisms. Then for any $s\in H^0(S,L)$ and $t\in H^0(S,P)$
$$ \delta_{LP}(st)|_S = s \cdot \delta_P(t)|_S+ t\cdot \delta_L(s)|_S.
$$\label{homol}
\end{lemma}
To prove this, let us first describe the connecting homomorphism $\delta:H^0(S,E)\to H^1(X,EM^\ast)$ associated to the short exact sequence 
$$ 0\to EM^\ast\to E\to E|_S\to 0,$$
for an arbitrary vector bundle $E$ on $X$. Let $\sU$ be a sufficiently fine Leray cover of $X$ and let the functions $f_i:U_i\to \cx$ cut out $S_i=S\cap U_i$ on each $U_i\in \sU$ (thus $f_i$ is a section of $M|_{U_i}$). Fix a trivialisation of $E$ on each $U_i$ and let $\psi_{ij}$ be the transition functions. Let $e\in H^0(S,E)$ and let $e_i$ represent $e$ on $S_i$ with respect to these trivialisations. Finally, let $\tilde e_i:U_i\to \cx$ be a holomorphic extension of $e_i$.
Then  $\delta (e)$ is represented on each $U_i\cap U_j$ by 
\begin{equation} \delta(e)_{ij}= \frac{\tilde e_i -\psi_{ij}\tilde e_j}{f_i}.\label{eq4}\end{equation}
Let us now trivialise $L$ and $P$ and let $l_{ij}$, $p_{ij}$ be the corresponding transition functions. Let  $s_i$ represents $s$ and $t_i$ represent $s$ and $t$, and $\tilde s_i, \tilde t_i$ be holomorphic extensions. Then, according to \eqref{eq4}:
$$ \delta_{LP}(st)_{ij}=\frac{\tilde s_i\tilde t_i -l_{ij}p_{ij}\tilde s_j\tilde t_j}{f_i}
=\tilde s_i \frac{\tilde t_i -p_{ij}\tilde t_j}{f_i}+ p_{ij}\tilde t_j\frac{\tilde s_i -l_{ij}\tilde s_j}{f_i},
$$
and the lemma follows by restricting to $S_i\cap S_j$.

\medskip

We return to the proof of the proposition, and apply the lemma with $X=Q$, $M=\sO(k,k)$, $L=\Lambda$, and $P=\sO(k-1,k-1)$. Since $PM^\ast=\sO(-1,-1)$ has no cohomology, $\delta_P=0$, and, so $\delta_{LP}(\phi t)|_S=t\cdot \delta_L(\phi)|_S$. The right-hand side is the connecting homomorphism of \eqref{E2}, and, therefore, we need to prove that $\delta_{LP}|_S$ is an isomorphism, i.e. that the connecting homomorphism $H^0(S,\Lambda(k-1,k-1))\to H^1(Q,\Lambda(-1,-1))$ followed by the restriction to $S$ is an isomorphism. We write the relevant short exact sequences as
\begin{equation} 0\to \Lambda(-1,-1)\to \Lambda(k-1,k-1)\to \Lambda|_S(k-1,k-1)\to 0,\label{E4}\end{equation}
\begin{equation*} 0\to \Lambda(-k-1,-k-1)\to \Lambda(-1,-1)\to \Lambda|_S(-1,-1)\to 0,\label{E5}\end{equation*}
and observe that, as $\Lambda=\sO(k,-k)$, the cohomology of both $\Lambda(k-1,k-1)$ and $\Lambda(-k-1,-k-1)$ vanish. Therefore, both the connecting homomorphism of \eqref{E4} and the restriction  $H^1(Q,\Lambda(-1,-1))\to H^1(S,\Lambda(-1,-1))$ are indeed isomorphisms.
\end{proof}

Thus, we can conclude, owing to Theorem \ref{twistor},  that $\sM_{k,0}$ has a natural strongly integrable $O(-1)$-structure. We are going to describe it explicitly and then show that it is nowhere pluricomplex (unless $k=1$).
\par
Let us recall, from \cite{A2,AM2}, a description of curves of bidegree $(k,k)$ satisfying the condition $\sO_S(k,-k)\simeq \sO_S$. Such a curve is determined by a pair of degree $k$ rational maps $f,g:\oP^1\to \oP^1$ and, as a set, $S$ is 
$$ S=\{(\zeta,\eta);\; f(\zeta)=g(\eta)\}.$$
If we write $f(\zeta)=p_1(\zeta)/q_1(\zeta)$, $g(\eta)=p_2(\eta)/q_2(\eta)$, where the $p_i$ and $q_j$ are degree $k$ polynomials, then a non-vanishing section of $\sO_S(k,-k)$ is given by $p_1(\zeta)/p_2(\eta)$ on $\{\zeta \neq \infty,\eta\neq \infty\}$. Curves, which are $\sigma$-invariant, are determined by pairs $(f,g)$ with $g(\eta)=\sigma f(\sigma(\eta))$. Such a curve automatically avoids the antidiagonal $\ol\Delta$ in $\oP^1\times \oP^1$. The manifold $\sM_{k,0}$ can be now identified with $\{${\em real curves $+$ real trivialisation of $L^{k}|_S$$\}$}, or, equivalently, with
pairs $(p,q)^T$ of coprime polynomials of degree $k$  modulo the action of $SU(2)$ on the vector $(p,q)^T$. The complexification $\sM^\cx$ of $\sM_{k,0}$, arising  as in Theorem \ref{twistor} (as a component of the Douady space $H_Z$) is now identified with the space of matrices
$$ A=\begin{pmatrix} p_1(\zeta) & p_2(\eta)\\ q_1(\zeta) & q_2(\eta)\end{pmatrix}$$
of polynomials of degree $k$, such that $(p_1,q_1)=(p_2,q_2)=1$, modulo left multiplication by $GL(2,\cx)$.
\par
As in the proof  of Theorem \ref{twistor}, for any $\zeta_0\in \oP^1$, we can identify the complex structure $J_{\zeta_0}$ by describing leaves of a complex foliation on $M^\cx$: for a curve $S$ and a section $\phi$ corresponding to a point of $\sM^\cx$, we need to identify all curves passing through $\phi(S)\cap \pi^{-1}(\{\zeta_0\}\times \oP^1)$, where $\pi:Z\to Q$ is the projection. Let $(S,\phi)$ correspond to a matrix $A$ of rational maps as above and write $f(\zeta)=p_1(\zeta)/q_1(\zeta)$, $g(\eta)=p_2(\eta)/q_2(\eta)$. Let $\eta_1,\dots,\eta_k$ be the points for which $g(\eta_i)=f(\zeta_0)$. We assume for convenience that they are distinct and that $f(\zeta_0)\neq 0,\infty$. Let $p_2(\eta_i)=\lambda_i p_1(\zeta_0)$, $i=1,\dots,k$. Then the leaf of the foliation passing through $(S,\phi)$ consists of all 
$$ Y=\begin{pmatrix} P_1(\zeta) & P_2(\eta)\\ Q_1(\zeta) & Q_2(\eta)\end{pmatrix}$$
such that $P_2(\eta_i)=\lambda_i P_1(\zeta_0)$ and $Q_2(\eta_i)=\lambda_i Q_1(\zeta_0)$, for $i=1,\dots,k$. Therefore 
$$P_2(\eta)=\frac{P_1(\zeta_0)}{p_1(\zeta_0)}p_2(\eta),\enskip 
Q_2(\eta)=\frac{Q_1(\zeta_0)}{q_1(\zeta_0)}q_2(\eta).$$
If $(S,\phi)$ is real, then the space $V^{1,0}_{\zeta_0}$ of vectors of type $(1,0)$ for $J_{\zeta_0}$ at this point is identified with the tangent space to this leaf. Thus, it follows that the leaf is determined by $(P,Q)=(P_1,Q_1)$ and, consequently,  
$$ H^0(S,N(-2,0))\simeq V^{1,0}_{\zeta_0}\cap V^{1,0}_{\zeta_1}\simeq \left\{(P,Q);\; \frac{P(\zeta_0)}{p_1(\zeta_0)}= \frac{P(\zeta_1)}{p_1(\zeta_1)},\; \frac{Q(\zeta_0)}{q_1(\zeta_0)}= \frac{Q(\zeta_1)}{q_1(\zeta_1)}\right\}/GL(2,\cx),$$
which has dimension $2k-2$. In fact, this argument shows that the vector bundle   $V^\cx/ K^\ast \sV^{1,0}$  on $\oP^1$ splits as
$$ \sO(k)\oplus\sO(k)\oplus \sO^{\oplus 2k-2}.$$

\appendix
\section{Strong regularity of characteristic sheaves}

\begin{proposition} Let $\sF$ be the characteristic sheaf of a  pluricomplex structure. Then
$$ h^0(\sF(p,q))=h^1(\sF(p,q))=0,$$
for every $p,q\in \oZ$ with $p+q=-2$.\label{p+q=-2}\end{proposition}
\begin{remark} This can be rephrased as follows. {\em Let $\sG$ be a  $\sigma$-sheaf on $\oP^1\times \oP^1$, of 
 dimension $1$, and such that $\ol\Delta\cap \supp \sG=\emptyset$. Suppose that the cohomology of $\sG$, $\sG(-1,1)$ and $\sG(1,-1)$ all vanish. Then $H^\ast(\sG(n,-n))=0$ for every $n\in \oZ$.}

\medskip

As the example of $\sG=\sO_S(2,-2)$ on a generic elliptic curve $S$ in $\oP^1\times \oP^1$ shows, the assumption that $\sG$ is a $\sigma$-sheaf is essential.\end{remark}

\begin{proof}
Let $m=p+1$, so that $\sF(p,q)=\sF(m-1,-m-1)$. Due to the $\sigma$-invariance, it is enough to prove the result for $m\geq 0$. Also, we have already proved it (in Proposition \ref{S} and Lemma \ref{-2,0}) for $m=0,1$. Tensoring \eqref{F} with $\sO(m-1,-m-1)$ yields a short exact sequence
\begin{multline*} 0\to \sO(m-2,-m-1)\otimes \cx^n\oplus \sO(m-1,-m-2)\otimes \cx^n \stackrel{M}{\longrightarrow} \\
\stackrel{M}{\longrightarrow}\sO(m-1,-m-1)\otimes \cx^{2n}\to \sF(m-1,-m-1)\to 0,\end{multline*}
of sheaves on $\oP^1\times \oP^1$. We observe that $H^0$ and $H^2$ of the first two terms vanish, so that it is enough to prove that the map $M(\zeta,\eta)$ induces an isomorphism on first cohomology groups. Using K\"unneth theorem we identify the first cohomology groups as
$$ W_1=H^1(\oP^1\times \oP^1,\sO(m-2,-m-1))=H^0(\oP^1,\sO(m-2))\otimes H^1 (\oP^1,\sO(-m-1)),$$
$$  W_2=H^1(\oP^1\times \oP^1,\sO(m-1,-m-2))=H^0(\oP^1,\sO(m-1))\otimes H^1 (\oP^1,\sO(-m-2)),$$
$$  W_3=H^1(\oP^1\times \oP^1,\sO(m-1,-m-1))=H^0(\oP^1,\sO(m-1))\otimes H^1 (\oP^1,\sO(-m-1)).$$
Let us write $M$ as 
$$ M(\zeta,\eta)=\begin{pmatrix} A & B\\ C & D\end{pmatrix},$$
where $A$ and $C$ are of degree $(1,0)$, $B$ and $D$ of degree $(0,1)$. Then the induced map $\hat M$ on first cohomology groups
\begin{equation} \hat M: W_1\otimes \cx^n\oplus W_2\otimes \cx^n \to W_3\otimes \cx^{2n}\label{hat M}\end{equation}
can be written in the matrix form (writing first  $W_3\otimes \cx^{2n}\simeq W_3\otimes \cx^{n}\oplus W_3\otimes \cx^{n}$)
$$ \hat M= \begin{pmatrix} \hat A & \hat B\\ \hat C & \hat D\end{pmatrix},$$
where 
$$\hat A=A\otimes 1: \bigl(H^0(\oP^1,\sO(m-2))\otimes \cx^n\bigr)\otimes H^1 (\oP^1,\sO(-m-1))\to W_3\otimes \cx^n,$$
$$\hat B= 1\otimes B: H^0(\oP^1,\sO(m-1))\otimes \bigl(H^1 (\oP^1,\sO(-m-2))\otimes \cx^n\bigr)\to W_3\otimes \cx^n,$$
and similarly $\hat C=C\otimes 1$, $\hat D=1\otimes D$. We need to make explicit the map $H^1(\oP^1,\sO(-m-2))\otimes \cx^n \to H^1(\oP^1,\sO(-m-1))\otimes \cx^n$ induced by a matrix $R$ of polynomials of degree $1$. Using Serre duality, we want to express this as a map 
$$ \hat R:H^0(\oP^1,\sO(m))\otimes \cx^n\to H^0(\oP^1,\sO(m-1))\otimes \cx^n.$$ Let us write $R=R_0+\eta R_1$. Then $\hat R$ applied to $v=v_0+v_1\eta+\dots +v_m\eta^m\in H^0(\oP^1,\sO(m))\otimes \cx^n$
gives 
\begin{equation} \hat R(v)= R_0(v-v_m\eta^m)+\frac{1}{\eta}R_1(v-v_0).\label{R}\end{equation}
Writing $V_{p,q}$ for the vector space of polynomials of bidegree $(p,q)$, we can view $\hat M$ as a map from 
$$\hat M: V_{m-2,m-1}\otimes \cx^n \oplus V_{m-1,m}\otimes \cx^n\longrightarrow V_{m-1,m-1}\otimes \cx^n\oplus V_{m-1,m-1}\otimes \cx^n,$$  
which can be expressed in the matrix form \eqref{hat M}, where $\hat A,\hat C$ denote multiplication by the matrices $A,C$ of polynomials linear in $\zeta$ and  $\hat B,\hat D$ act on $v(\zeta,\eta)=v_0(\zeta)+v_1(\zeta)\eta+\dots +v_m(\zeta)\eta^m$ as in \eqref{R}. 
\par
We now assume that $M$ is as in Remark \ref{M}, so that $A=X+\zeta Y$, $B=-1$, $C=\zeta$, $D=-\ol Y+\eta \ol X$. Suppose that $(u,v)\in V_{m-2,m-1}\otimes \cx^n \oplus V_{m-1,m}\otimes \cx^n$ belongs to $\Ker\hat M$. From the above, this is equivalent to
$$ (X+\zeta Y)u-(v-v_m\eta^m)=0, \quad \zeta u -\bar Y(v-v_m\eta^m)+\frac{1}{\eta}\bar X(v-v_0).$$
Therefore
$$ (X+\zeta Y)\bigl(\bar Y(v-v_m\eta^m)-\frac{1}{\eta}\bar X(v-v_0)\bigr)=\zeta(v-v_m\eta^m),$$
and consequently (recalling that $X$ is invertible)
$$ v_{i+1}(\zeta)= (\bar X)^{-1}\bigl( (X+\zeta Y)\bar Y-\zeta\bigr)v_i(\zeta), \quad i=0,1,\dots, m-1. $$
Hence $v_0=0$ implies that $v=0$. On the other hand, the coefficient of $\zeta$ in this operator acting on $v_i$ is $Y\bar Y-1$, which has null kernel thanks to Proposition \ref{matrices}. Therefore,
if $v_0\neq 0$, then $v_m$ must have degree at least $m$ in $\zeta$, which contradicts the assumption that $v$ is of bidegree $(m-1,m)$. Hence $\Ker \hat M=0$ and we are done.
\end{proof}

\end{document}